\documentclass[oneside,english]{amsart}

\usepackage[T1]{fontenc}
\usepackage{color}
\usepackage{bm}
\usepackage{amstext}
\usepackage{amsthm}
\usepackage{amssymb}
\usepackage{esint}
\usepackage[all]{xy}

\makeatletter
\numberwithin{equation}{section}
\numberwithin{figure}{section}
\theoremstyle{plain}
\newtheorem{thm}{\protect\theoremname}[section]
  \theoremstyle{plain}
  \newtheorem{prop}[thm]{\protect\propositionname}
  \theoremstyle{remark}
  \newtheorem{rem}[thm]{\protect\remarkname}
  \theoremstyle{plain}
  \newtheorem{lem}[thm]{\protect\lemmaname}
  \theoremstyle{definition}
  \newtheorem{defn}[thm]{\protect\definitionname}
  \theoremstyle{plain}
  \newtheorem{conjecture}[thm]{\protect\conjecturename}

\newcounter{myparagraph}[subsection]
\newcommand{\myparagraph}{\refstepcounter{myparagraph}}
\renewcommand{\themyparagraph}{{\arabic{section}.\arabic{subsection}.\alph{myparagraph}}}
\newcommand*{\para}[1]{\vskip0.3cm\noindent\hspace{-3pt}\myparagraph{\bf (\themyparagraph){ {#1.}}}}
\usepackage[all]{xy}

\usepackage{color}
\usepackage{amsthm}
\usepackage{bm}
\usepackage{graphicx, color, ulem}

\usepackage[labelformat=empty]{caption}

 \newtheorem{cla}[thm]{Claim}

\makeatother

\usepackage{babel}
  \providecommand{\conjecturename}{Conjecture}
  \providecommand{\definitionname}{Definition}
  \providecommand{\lemmaname}{Lemma}
  \providecommand{\propositionname}{Proposition}
  \providecommand{\remarkname}{Remark}
\providecommand{\theoremname}{Theorem}

\begin{document}

\global\long\def\Ssp{Z_{2}^{sp}}
\global\long\def\Usp{\tilde{X}_{2}^{sp}}
\global\long\def\Xsp{\tilde{X}_{0}^{sp}}

\global\long\def\Sspp{Z_{1}^{sp}}
\global\long\def\Uspp{\tilde{X}_{1}^{sp}}

\global\long\def\Hsp{\tilde{X}_{0}^{sp,\sharp}}

\global\long\def\cXsp{\tilde{X}_{0}^{*}}

\global\long\def\Xspi{\tilde{X}_{0}^{sp,(1)}}
\global\long\def\Xspii{\tilde{X}_{0}^{sp,(2)}}
\global\long\def\Xspiii{\tilde{X}_{0}^{sp,(3)}}

\global\long\def\cvarphi{\check{\varphi}}
 \global\long\def\crho{\check{\rho}}

\global\long\def\rG{\mathrm{G}}

\global\long\def\sE{\mathcal{E}}
\global\long\def\sK{\mathcal{K}}

\global\long\def\sQ{\mathcal{Q}}

\global\long\def\sO{\mathcal{O}}

\global\long\def\mC{\mathbb{C}}

\global\long\def\mP{\mathbb{P}}

\global\long\def\mR{\mathbb{R}}

\global\long\def\mQ{\mathbb{Q}}

\global\long\def\mZ{\mathbb{Z}}

\global\long\def\Hom{\mathrm{Hom}}

\global\long\def\sZ{\mathcal{Z}}

\global\long\def\tx{{\tt x}}
\global\long\def\ty{{\tt y}}
\global\long\def\tz{{\tt z}}

\title{K3 surfaces from configurations of six lines in $\mathbb{P}^{2}$
and mirror symmetry I}

\author{Shinobu Hosono, Bong H. Lian, Hiromichi Takagi and Shing-Tung Yau}
\begin{abstract}
From the viewpoint of mirror symmetry, we revisit the hypergeometric
system $E(3,6)$ for a family of K3 surfaces. We construct a good
resolution of the Baily-Borel-Satake compactification of its parameter
space, which admits special boundary points (LCSLs) given by normal
crossing divisors. We find local isomorphisms between the $E(3,6)$
systems and the associated GKZ systems defined locally on the parameter
space and cover the entire parameter space. Parallel structures are
conjectured in general for hypergeometric system $E(n,m)$ on Grassmannians.
Local solutions and mirror symmetry will be described in a companion
paper \cite{HLTYpartII}, where we introduce a K3 analogue of the
elliptic lambda function in terms of genus two theta functions. 
\end{abstract}

\maketitle
{\small{}\tableofcontents{}}{\small \par}

\section{\textbf{\textup{Introduction}}}

Consider double covers of $\mathbb{P}^{1}$ branched along four points
in general positions. They define a family of elliptic curves called
the Legendre family over the moduli space of the configurations of
four points on $\mathbb{P}^{1}$, which are naturally parametrized
by the cross ratio of the four points. It is a classical fact that
the elliptic lambda function is defined as a modular function that
arises from the hypergeometric series representing period integrals
for the Legendre family. 

Higher dimensional analogues of the Legendre family have been studied
in many context in the history of modular forms and analysis related
to them. Among others, Matsumoto, Sasaki and Yoshida \cite{YoshidaEtal}
have studied extensively in the 90's the two dimensional generalization
of the Legendre family, i.e., the double covers of the projective
plane $\mathbb{P}^{2}$ branched along six lines in general positions.
After making suitable resolutions, the double covers define a family
of smooth K3 surfaces parametrized by the configurations of six lines.
In \cite{YoshidaEtal}, the authors studied in great details the period
integrals of the family and determined the monodromy properties of
the period integrals completely. They described the set of the differential
equations satisfied by the period integrals in terms of the so-called
Aomoto-Gel'fand system \cite{Ao,GelfandGraev,GelfandGelfand} on Grassmannians
$G(3,6)$, and named them hypergeometric system $E(3,6)$. 

Around the same time in the 90's, period integrals for families of
Calabi-Yau manifolds were studied intensively to verify several predictions
from mirror symmetry of Calabi-Yau manifolds. For Calabi-Yau manifolds
given as complete intersections in a toric variety, it is now known
that the period integrals for such a family are solutions to a hypergeometric
system called Gel'fand-Kapranov-Zelevinski (GKZ) system. In particular,
it was shown in \cite{HLY,HKTY} that for GKZ systems in this context
there exist special boundary points called large complex structure
limits (LCSLs), and mirror symmetry appears nicely in the form of
\textit{generalized} Frobenius method which provides a closed formula
for period integrals and mirror map near these boundary points. 

In this paper, we will revisit the hypergeometric system $E(3,6)$
from the viewpoint of mirror symmetry of K3 surfaces. Despite the
fact that many analytic properties of $E(3,6)$ have been studied
in details in the literature e.g. \cite{YoshidaEtal,Sekiguchi-Taka},
it was not clear how to construct the degeneration points (LCSLs)
in the parameter space of $E(3,6)$. We will find that the $\mathcal{D}$-module
associated to the hypergeometric system $E(3,6)$ over its parameter
space is locally trivialized by the $\mathcal{D}$-module of the corresponding
GKZ hypergeometric system (\textbf{Theorem \ref{thm:Thm1}}). Thanks
to this general property, it turns out that the techniques developed
in \cite{HLY,HKTY} for GKZ systems can be applied to $E(3,6)$ (\textbf{Theorem
\ref{thm:Thm2}}); this includes the existence of the degeneration
points and the closed formula of the period integrals around them.
To show our results, we first cover the parameter space of $E(3,6)$,
which can be identified with the Baily-Borel-Satake compactification
of the family of the K3 surfaces, by certain Zariski open subsets
of toric varieties on which GKZ systems are defined. Using this covering
property, we finally show that there are two nice algebraic resolutions
of the Baily-Borel-Satake compactification (\textbf{Theorem \ref{thm:Resolutions}})
which are related by a four dimensional flip. 

Around the special degeneration points (LCSLs), following \cite{HLY,HKTY},
we can define the so-called mirror maps. In our case, these mirror
maps can be regarded as two dimensional generalizations of the elliptic
lambda function. We will call them \textit{$\lambda_{K3}$-functions}.
In a companion paper \cite{HLTYpartII}, we will describe the $\lambda_{K3}$-functions
in terms of genus two theta functions. Moreover, we will find that,
corresponding to the two different algebraic resolutions related by
a flip, there exist two different definitions for the $\lambda_{K3}$-functions. 

\vskip0.2cm

Here is the outline of this paper: In Section 2, after introducing
our family of K3 surfaces and the hypergeometric system $E(3,6)$
satisfied by period integrals, we will introduce the configuration
space of six ordered points as the parameter space of $E(3,6)$. We
summarize known properties about the compactification of the parameter
space of $E(3,6)$ and also introduce other closely related parameter
spaces: the configuration space of 3 points and 3 lines in $\mathbb{P}^{2}$
and the parameter space of the GKZ system which trivializes the $E(3,6)$.
In Section 3, we describe a toric compactification of the parameter
space of this GKZ system, and construct the expected LCSLs after making
a resolution. In Section 4, we observe that the configuration space
of 3 points and 3 lines in $\mathbb{P}^{2}$ arises naturally from
certain residue calculations of a period integral. We find that the
toric compactification for the GKZ system gives a toric partial resolution
of the GIT compactification of the configuration space of 3 points
and 3 lines in $\mathbb{P}^{2}$. In Section 5 and 6, we reconstruct
the partial resolution using classical projective geometry. Transforming
this partial resolution (locally) by certain birational map to the
Baily-Borel-Satake compactification, we construct the desired algebraic
resolutions of the Baily-Borel-Satake compactification. In Section
7, we combine the results of the preceding sections and rephrase them
in the language of $\mathcal{D}$-modules to state the main results
of this paper. We also formulate conjectural generalizations of our
results. 

\vskip0.3cm

\noindent\textbf{Acknowledgements:} This project was initiated by
a question by Naichung C. Leung about hypergeometric systems on Grassmannians
to S.H. We are grateful to him for asking the question which actually
has drawn our attention to the problems left unsolved in the 90's.
We are also grateful to Osamu Fujino for useful advice to the proof
of Claim \ref{cla:flip}. S.H. would like to thank for the warm hospitality
at the CMSA at Harvard University where progress was made. This work
is supported in part by  Grant-in Aid Scientific Research (C 16K05105,
S 17H06127, A 18H03668 S.H. and C 16K05090 H.T.). B.H.L and S.-T.
Yau are supported by the Simons Collaboration Grant on Homological
Mirror Symmetry and Applications 2015--2019.

~

~

\section{\textbf{\textup{The hypergeometric system $E(3,6)$}} }

\subsection{Double covering of $\mathbb{P}^{2}$ branched along six lines }

Let us consider six lines $\ell_{i}(i=1,..,6)$ in $\mathbb{P}^{2}$
in general position. We denote them by $\left\{ \ell_{i}(\tx,\ty,\tz)=0\right\} \subset\mathbb{P}^{2}$
with the following linear forms: 
\[
\ell_{i}(\tx,\ty,\tz):=a_{0i}\tz+a_{1i}\tx+a_{2i}\ty\;\;(i=1,...,6).
\]
When the lines are in general position, the double cover branched
along the six lines defines a singular K3 surface with $A_{1}$ singularity
at each 15 intersection points $P_{ij}:=\ell_{i}\cap\ell_{j}$. Blowing-up
the 15 $A_{1}$ singularities, we have a smooth K3 surface $X$ of
Picard number 16 generated by the hyperplane class $H$ from $\mathbb{P}^{2}$
and the $-2$ curves of the exceptional divisors $E_{ij}$ from the
blow-up. The configurations of six lines define a four dimensional
family of K3 surfaces,\textcolor{red}{{} }which we will call double
cover family of K3 surfaces for short in this paper. The period integrals
of the family of holomorphic two forms and their monodromy properties
were studied extensively in \cite{YoshidaEtal} by analyzing the hypergeometric
system $E(3,6)$. We will revisit the system $E(3,6)$ from the viewpoint
of mirror symmetry and provide a new perspective for mirror symmetry.

\subsection{Period integrals of $X$}

Recall that the Legendre family consists of elliptic curves given
by double covers of $\mathbb{P}^{1}$ branched along four points in
general position. The double cover family of K3 surfaces is a natural
generalization of the Legendre family. Analogous to the period integrals
of the Legendre family \cite[Chap.IV,10]{YoshidaBook} are the period
integrals of a holomorphic two form: 
\begin{equation}
\bar{\omega}_{C}(a)=\int_{C}\frac{d\mu}{\sqrt{\prod_{i=1}^{6}\ell_{i}(\tx,\ty,\tz)}},\label{eq:peridIntI}
\end{equation}
where $d\mu=\tx d\ty\wedge d\tz-\ty d\tx\wedge d\tz+\tz d\tx\wedge d\ty$
and $C$ is an integral (transcendental) cycle in $H_{2}(X,\mathbb{Z}$).
Explicit descriptions of the transcendental cycles can be found in
\cite{YoshidaEtal}. Also the lattice of transcendental cycles is
determined to be 
\[
T_{X}\simeq U(2)\oplus U(2)\oplus A_{1}\oplus A_{1},
\]
where $\mbox{\ensuremath{U}(2)}$ represents the hyperbolic lattice
$U$ of rank 2 with the Gram matrix multiplied by 2, and $A_{1}=\langle-2\rangle$
is the root lattice of $sl(2,\mathbb{C})$. As obvious in the above
definition, the period integrals $\bar{\omega}_{C}(a)$ determine
(multi-valued) functions defined on the set of $3\times6$ matrices
$A$ representing (ordered) six lines in general positions. Explicitly,
we describe the matrices $A$ by 
\begin{equation}
A=\left(\begin{matrix}a_{01} & a_{02} & a_{03} & a_{04} & a_{05} & a_{06}\\
a_{11} & a_{12} & a_{13} & a_{14} & a_{15} & a_{16}\\
a_{21} & a_{22} & a_{23} & a_{24} & a_{25} & a_{26}
\end{matrix}\right).\label{eq:matrix-A}
\end{equation}
Let $M_{3,6}$ be the affine space of all $3\times6$ matrices, and
set 
\[
M_{3,6}^{o}:=\left\{ A\in M_{3,6}\mid D(i_{1},i_{2},i_{3})\not=0\,(1\leq i_{1}<i_{2}<i_{3}\leq6)\right\} 
\]
 with $D(i_{1},i_{2},i_{3})$ representing $3\times3$ minors of $A$.
Then, under the genericity assumption, the configurations of six lines
are parametrized by 
\[
P(3,6):=GL(3,\mathbb{C})\diagdown M_{3,6}^{o}\diagup(\mathbb{C}^{*})^{6},
\]
where $(\mathbb{C}^{*})^{6}$ represents the diagonal $\mathbb{C}^{*}$-actions.
The differential operators which annihilate the period integrals define
the hypergeometric system of type $E(3,6)$ \cite[Sect.1.4]{YoshidaEtal},
which is the Aomoto-Gel'fand system on Grassmannian $G(3,6)$ \cite{Ao,GelfandGraev,GelfandGelfand}.
The following proposition is easy to derive.
\begin{prop}
The period integral $\bar{\omega}(a)$ satisfies the following set
of differential equations:
\begin{equation}
\begin{aligned}(\mathrm{i}) & \;\;\sum_{i=0}^{2}a_{ij}\frac{\partial\;}{\partial a_{ij}}\bar{\omega}(a)=-\frac{1}{2}\bar{\omega}(a), & 1\leq j\leq6,\qquad\qquad\qquad\\
(\mathrm{ii}) & \;\;\sum_{j=1}^{6}a_{ij}\frac{\partial\;}{\partial a_{kj}}\bar{\omega}(a)=-\delta_{ik}\bar{\omega}(a) & 1\leq i\leq3,\qquad\qquad\qquad\\
(\mathrm{iii}) & \;\;\frac{\partial^{2}\;}{\partial a_{ij}\partial a_{kl}}\bar{\omega}(a)=\frac{\partial^{2}\;}{\partial a_{il}\partial a_{kj}}\bar{\omega}(a), & \;\;\;1\leq i,k\leq3,\;1\leq j,l\leq6.
\end{aligned}
\label{eq:GelfandEq}
\end{equation}
\end{prop}
\begin{proof}
The relations (i) and (iii) are rather easy to verify by differentiating
(\ref{eq:peridIntI}) directly. To derive (ii), we note that $d\mu=i_{E}d\tx\wedge d\ty\wedge d\tz$
holds with the Euler vector field $E=\tx\frac{\partial\;}{\partial\tx}+\ty\frac{\partial\;}{\partial\ty}+\tz\frac{\partial\;}{\partial\tz}$.
Since the Euler vector field is invariant under the linear coordinate
transformation, it is easy to verify 
\[
\bar{\omega}(ga)=(\det g)^{-1}\bar{\omega}(a),
\]
for the left $GL(3,\mathbb{C})$-action on $A=(a_{ij})\in M_{3,6}^{o}$.
The relation follows from the infinitesimal form of this relation. 
\end{proof}
In the paper \cite{YoshidaEtal}, the hypergeometric functions representing
the period integrals has been studied in details using the following
affine coordinate system of the quotient $P(3,6)$:
\[
\left(\begin{matrix}1 & 0 & 0 & 1 & 1 & 1\\
0 & 1 & 0 & 1 & x_{1} & x_{2}\\
0 & 0 & 1 & 1 & x_{3} & x_{4}
\end{matrix}\right).
\]
However, this affine coordinate turns out to be inadequate for studying
mirror symmetry. In particular, in order to construct the special
boundary points, called large complex structure limits (LCSLs), we
need a suitable compactification.

\subsection{Period domain and compactifications of the parameter space $P(3,6)$ }

Mirror symmetry for two or three dimensional Calabi-Yau hypersurfaces
or complete intersections in toric Fano varieties was worked out in
many examples in the 90's by constructing families of Calabi-Yau manifolds
and by studying period integrals associated to holomorphic $n$-forms
for $n=2$ or $3$. It is now known that the geometry of mirror symmetry
appears, in a certain simplified form \cite{SYZ}, near the special
boundary points which are given as normal crossing boundary divisors
in suitable compactifications of the parameter spaces for the families
of hypersurfaces \cite{HLY}. The double cover family of K3 surfaces
does not belong to these well-studied families of Calabi-Yau manifolds.
However, its parameter space $P(3,6)$ admits many nice compactifications
relevant to describe the boundary points. We summarize several compactifications
and describe their relationships.

\para{Period domain ${\mathcal D}_{K3}=\Omega(U(2)^{\oplus2}\oplus A_1^{\oplus2})$}
Since the generic member $X$ of the double cover family of K3 surfaces
has the transcendental lattice $T_{X}\simeq U(2)^{\oplus2}\oplus A_{1}^{\oplus2},$
the period integral defines a map from $P(3,6)$ to the period domain
\[
\mathcal{D}_{K3}:=\left\{ [\omega]\in\mathbb{P}((U(2)^{\oplus2}\oplus A_{1}^{\oplus2})\otimes\mathbb{C})\mid\omega.\omega=0,\omega.\bar{\omega}>0\right\} ^{+},
\]
where $^{+}$ represents one of the connected components. Let us denote
by $G$ the Gram matrix of the lattice $U(2)^{\oplus2}\oplus A_{1}^{\oplus2}$
given in the following block-diagonal from: 
\[
G=\left(\begin{matrix}0 & 2\\
2 & 0
\end{matrix}\right)\oplus\left(\begin{matrix}0 & 2\\
2 & 0
\end{matrix}\right)\oplus(-2)\oplus(-2).
\]
Using this, we define 
\[
\mathcal{G}:=\left\{ g\in PGL(6,\mathbb{Z})\mid\,^{t}gGg=G,H(g)>0\right\} 
\]
with $H(g)=(g_{11}+g_{12})(g_{33}+g_{34})-(g_{13}+g_{14})(g_{31}+g_{32})$,
which is a discrete subgroup of $\mathrm{Aut}(\mathcal{D}_{K3}$)
(see \cite[Sect.1.4]{Matsu93}). In \cite[Prop.2.7.3]{YoshidaEtal},
it is shown that the monodromy group of period integrals coincides
with the congruence subgroup $\mathcal{G}(2)=\left\{ g\in\mathcal{G}\mid g\equiv E_{6}\text{ mod }2\right\} $,
hence $P(3,6)\simeq\mathcal{D}_{K3}/\mathcal{G}(2)$ holds and $\mathcal{D}_{K3}$
gives the unifomization of the multi-valued period integral on the
configuration space $P(3,6)$. 

\para{GIT compactification ${\mathcal M}_6$} \label{para:GIT-Y}
A natural compactification of $P(3,6)$ is given by parameterizing
the six lines $\left\{ \ell_{i}\right\} $ by the corresponding points
$\left\{ \bm{a}_{i}\right\} $ in the dual projective space $\check{\mP}^{2}$
and arrange the corresponding ordered six points as in (\ref{eq:matrix-A})
with $\bm{a}_{i}=\,^{t}(a_{1i},a_{2i},a_{3i})$. The configuration
space of these ordered six points is a well-studied object in geometric
invariant theory. In \cite{DoOrt,Reuv}, one can find that a compactification
$\mathcal{M}_{6}$ is given as a double cover of $\mathbb{P}^{4}$
branched along the so-called Igusa quartic, which has the following
description: 
\begin{equation}
\mathcal{M}_{6}\simeq\left\{ Y_{5}^{2}=F_{4}(Y_{0},...,Y_{4})\right\} \subset\mathbb{P}(1^{5},2),\label{eq:DoubleCover}
\end{equation}
where $F_{4}$ is the quartic polynomial 
\[
F_{4}=(Y_{0}Y_{s}+Y_{2}Y_{3}-Y_{1}Y_{4})^{2}+4Y_{0}Y_{1}Y_{4}Y_{s}
\]
with $Y_{s}:=Y_{0}-Y_{1}+Y_{2}+Y_{3}-Y_{4}$. See Appendix \ref{sec:Appendix-birat-M}.1
for a brief summary. Since $\mathcal{M}_{6}$ is a geometric compactification,
the (multi-valued) period map from $P(3,6)$ to $\mathcal{D}_{K3}$
naturally extends to $\mathcal{M}_{6}$, which we will write $\mathcal{P}:\mathcal{M}_{6}\to\mathcal{D}_{K3}$. 

\para{Baily-Borel-Satake compactifications} In \cite{Matsu93}, it
was shown explicitly that the double cover $\mathcal{M}_{6}$ coincides
with the Baily-Borel-Stake compactification of certain arithmetic
quotient of the symmetric space of type $I_{2,2}$ defined by 
\[
\mathbb{H}_{2}=\left\{ W\in Mat(2,\mathbb{C})\mid(W^{\dagger}-W)/2i>0\right\} ,
\]
where $W^{\dagger}:=\,^{t}\overline{W}$. The Siegel upper half space
of genus two $\frak{h}_{2}$ is contained in $\mathbb{H}_{2}$ as
the locus satisfying $W=\,^{t}W$ . To introduce the arithmetic quotients
of $\mathbb{H}_{2}$, following the notation of \cite{Matsu93}, we
define discrete subgroups of $\mathrm{Aut}(\mathbb{H}_{2})$:
\[
\begin{aligned} & \Gamma:=\left\{ g\in PGL(4,\mathbb{Z}[i])\mid g^{\dagger}Jg=J\right\} ,\;g^{\dagger}:=\,^{t}\bar{g},\,\,J:=\left(\begin{smallmatrix}0 & E_{2}\\
-E_{2} & 0
\end{smallmatrix}\right),\\
 & \Gamma_{T}:=\Gamma\rtimes\langle T\rangle,\quad T:W\mapsto\,^{t}W\,\,(W\in\mathbb{H}_{2}),\\
 & \Gamma_{M}:=\left\{ gT^{a}\in\Gamma_{T}\mid(-1)^{a}\det(g)=1,a=0,1\right\} .
\end{aligned}
\]
We also introduce the congruence subgroups: 
\[
\begin{aligned} & \Gamma(1+i):=\left\{ g\mid g\equiv E_{4}\text{ mod (1+i)}\right\} ,\\
 & \Gamma_{T}(1+i):=\Gamma(1+i)\rtimes\left\langle T\right\rangle .
\end{aligned}
\]
Then the arithmetic group relevant to the quotient is
\[
\begin{aligned}\Gamma_{M}(1+i):=\Gamma_{M}\cap\Gamma_{T}(1+i),\end{aligned}
\]
which defines the quotient $\Gamma_{M}(1+i)\setminus\mathbb{H}_{2}$.
Note that there is a natural map $\Gamma_{M}(1+i)\setminus\mathbb{H}_{2}\to\Gamma_{T}(1+i)\setminus\mathbb{H}_{2}$
which is generically $2:1$. 

The Baily-Borel-Satake compactification of the arithmetic quotient
of $\Gamma_{T}(1+i)\setminus\mathbb{H}_{2}$ is given explicitly by
the Zariski closure of the image of the map 
\[
\Phi:\mathbb{H}_{2}\to\mathbb{P}^{9},\;W\mapsto[\Theta_{1}(W)^{2},\cdots,\Theta_{10}(W)^{2}],
\]
where theta functions $\Theta_{i}(W)(i=1,...,10)$ correspond to ten
different (even) spin structures. These squares of the theta functions
are modular forms of weight two on the group $\Gamma_{T}(1+i)$ with
a character given by determinant $\det(gT^{a})=\det(g)$ for $gT^{a}\in\Gamma_{T}(1+i)$,
$(a=0,1)$ (see \cite[Prop. 3.1.1]{Matsu93}). Also, there are five
linear relations among them. Hence we have $\overline{\Gamma_{T}(1+i)\setminus\mathbb{H}_{2}}\simeq\mathbb{P}^{4}$
for the compactification. When $W=\,^{t}W$, these theta functions
reduces to the theta functions $\theta_{1}(\tau)^{4},...,\theta_{10}(\tau)^{4}$
of genus two which generate Siegel modular forms of level two and
even weights. The Igusa quartic is a quartic relation satisfied by
$\theta_{i}(\tau)^{4}$, hence defines a quartic hypersurface in $\mathbb{P}^{4}$.

Actually the above five linear relations correspond to Pl\"ucker
relations (\ref{eq:AppC-rel-Ys}) under a suitable identification
of the $\Theta_{i}(W)^{2}$'s with the semi-invariants $Y_{k}$'s,
which we will do in our companion paper \cite{HLTYpartII} to introduce
$\lambda_{K3}$-functions. Under this identification, the Igusa quartic
$\left\{ F_{4}(Y_{0},...,Y_{4})=0\right\} \subset\mathbb{P}(1^{5})$
above coincides with the closure of $\Phi(\left\{ W=\,^{t}W\right\} )$.

To describe further relations of the arithmetic quotients to $\mathcal{M}_{6}$
in (\ref{eq:DoubleCover}), we note an isomorphism $\mathcal{D}_{K3}\simeq\mathbb{H}_{2}$
of the two domains (see \cite[Sect.1.3]{Matsu93}). Here we also note
the isomorphism $\mathcal{G}(2)\simeq\Gamma_{M}(1+i)$ \cite[Prop.1.5.1]{Matsu93}.
Due to the former isomorphism, we have the period map $\mathcal{P}:\mathcal{M}_{6}\to\mathbb{H}_{2}$
as a multi-valued map on $\mathcal{M}_{6}$ with its monodromy group
$\mathcal{G}(2)$. 
\begin{prop}[{ \cite[Thm.4.4.1]{Matsu93}}]
We have the following commutative diagram:
\[
\xymatrix{\mathcal{M}_{6}\;\;\ar[dr]_{\Phi_{Y}}\ar[rr]^{\mathcal{P}} &  & \;\;\mathbb{H}_{2}\ar[dl]^{\Phi}\\
 & \mathbb{P}^{9},
}
\]
where $\Phi_{Y}$ defined by $A\mapsto[Y_{0}(A),...,Y_{4}(A),Y_{6}(A),...,Y_{10}(A)]$
with the semi-invariants of $3\times6$ matrices given in Appendix
\ref{sec:Appendix-birat-M}.1. 
\end{prop}
The map $\Phi_{Y}:\mathcal{M}_{6}\to\mathbb{P}^{4}\subset\mathbb{P}^{9}$
is $2:1$ whose branch locus is the Igusa quartic $\left\{ F_{4}(Y)=0\right\} $
in $\mathbb{P}(1^{5})\simeq\overline{\Gamma_{T}(1+i)\setminus\mathbb{H}_{2}}$.
On the other hand, we have noted that there is a natural map $\Gamma_{M}(1+i)\setminus\mathbb{H}_{2}\to\Gamma_{T}(1+i)\setminus\mathbb{H}_{2}$
which is generically $2:1$. All these facts are unified by the existence
of new theta function $\Theta$ which is modular of weight 4 on $\Gamma_{T}$
\cite[Lem.3.1.3]{Matsu93}, and which vanishes on $\left\{ W=\,^{t}W\right\} $.
Proofs of the following results can be found in Proposition 3.1.5
and Theorem 3.2.4 in \cite{Matsu93}. 
\begin{prop}
The theta functions satisfy 
\begin{equation}
\Theta(W)^{2}=\frac{3\cdot5^{2}}{2^{6}}\bigg\{\bigg(\sum_{i=1}^{10}\Theta_{i}(W)^{4}\bigg)^{2}-4\,\sum_{i=1}^{10}\Theta_{i}(W)^{8}\bigg\}\label{eq:Theta-equation}
\end{equation}
and this describes the Baily-Borel-Satake compactification of $\Gamma_{M}(1+i)\setminus\mathbb{H}_{2}$
as a double cover of $\overline{\Gamma_{T}(1+i)\setminus\mathbb{H}_{2}}\simeq\mathbb{P}^{4}$.
This compactification is isomorphic to the GIT compactification $\mathcal{M}_{6}$. 
\end{prop}
The geometry of the double cover (\ref{eq:DoubleCover}), or (\ref{eq:Theta-equation}),
is a well-studied subject in many respects. For example, it is known
that the double cover is singular along 15 lines which are identified
with the one dimensional boundary component of the Baily-Borel-Satake
compactification. It is also singular at 15 points, which are given
as intersections of the lines, representing the zero dimensional components
of the Baily-Borel-Satake compactification. In Section \ref{sec:BBS},
we will describe the configuration of these singularities, and will
find a good resolution from the viewpoint of mirror symmetry. Our
resolution is also important for introducing the functions $\lambda_{K3}$
which are the mirror maps of our family. 

\para{Birational toric variety ${\mathcal M}_{3,3}$} The Aomoto-Gel'fand
system $E(3,6)$ should be considered as a hypergeometric system defined
over the GIT compactification (or the Baily-Borel-Satake compactification)
$\mathcal{M}_{6}$. In the next section, we will find that there appears
another variety $\mathcal{M}_{3,3}$, which is a toric variety, from
the analysis of period integrals. Classically, $\mathcal{M}_{3,3}$
comes from the following birational correspondence to $\mathcal{M}_{6}$
\cite{Reuv}. Let us consider the six lines $\left\{ \ell_{i}\right\} $
in general position and select three lines $\ell_{i_{1}},\ell_{i_{2}},\ell_{i_{3}}$
to have the map 
\begin{equation}
\left\{ \ell_{i}\right\} \mapsto\left\{ \ell_{i_{1}}\cap\ell_{i_{2}},\ell_{i_{1}}\cap\ell_{i_{3}},\ell_{i_{2}}\cap\ell_{i_{3}},\ell_{i_{4}},\ell_{i_{5}},\ell_{i_{6}}\right\} ,\label{eq:birat-map-M6-M33}
\end{equation}
which gives a configuration of three points in $\mathbb{P}^{2}$ and
three lines in $\mathbb{\mathbb{P}}^{2}$. This defines a rational
map from $\mathcal{M}_{6}$ to the moduli space of configurations
of three points and three lines in $\mathbb{P}^{2}.$ The variety
$\mathcal{M}_{3,3}$ is the GIT compactification of these configurations,
which turns out to be the following toric hypersurface; 
\[
\mathcal{M}_{3,3}\simeq\left\{ X_{1}X_{2}X_{3}=X_{4}X_{5}X_{6}\right\} \subset\mathbb{P}^{5}.
\]
Since three points in general position determines three lines passing
through them, given a configuration of three points and three lines
in general position, we have six lines in general position in $\mathbb{P}^{2}$.
Hence the map (\ref{eq:birat-map-M6-M33}) gives a birational map
between $\mathcal{M}_{6}$ and $\mathcal{M}_{3,3}$. See Appendix
\ref{sec:Appendix-birat-M} for its explicit form. This toric variety
$\mathcal{M}_{3,3}$ will play a key role in our analysis of period
integrals defined on $\mathcal{M}_{6}$.

\subsection{Toroidal compactification $\mathcal{M}_{SecP}$ of $P(3,6)$}

In this section, we shall apply the techniques in \cite{HLY} to give
a toric compactification of $P(3,6)$. This is essential for describing
mirror symmetry of the double cover family of K3 surfaces. The compactification
$\mathcal{M}_{6}$ of $P(3,6)$ deals with the $GL(3,\mathbb{C})$
action on the affine coordinates of $A=(a_{ij})$ in terms of classical
invariant theory. Similarly for the birational toric variety $\mathcal{M}_{3,3}$.
Our third compactification $\mathcal{M}_{SecP}$ arises from reducing
the group actions of $GL(3,\mathbb{C})$ and $(\mathbb{C}^{*})^{6}$
on $A\in M_{3,6}^{o}$ to the diagonal torus actions. 

\para{Partial 'gauge' fixing to $T\simeq (\mC^*)^5$} To reduce $GL(3,\mathbb{C})$
action to the diagonal torus actions, we transform the general matrix
$A\in M_{3,6}^{o}$ to the form, 
\begin{equation}
\left(\begin{matrix}1 & 0 & 0 & a_{2} & b_{1} & c_{0}\\
0 & 1 & 0 & a_{0} & b_{2} & c_{1}\\
0 & 0 & 1 & a_{1} & b_{0} & c_{2}
\end{matrix}\right)=:(E_{3}\,\bm{a}\,\bm{b}\,\bm{c}).\label{eq:E3-gaugeA}
\end{equation}
Clearly this reduces the $GL(3,\mathbb{C})$ action from the left
to the diagonal tori. We note that there are still residual group
actions of the diagonal tori $(\mathbb{C}^{*})^{3}\subset GL(3,\mathbb{C})$
combined with the $(\mC^{*})^{6}$ action from the right, i.e., 
\[
T:=\left\{ (g,t)\in GL(3,\mC)\times(\mathbb{C}^{*})^{6}\mid g\left(E_{3}\begin{smallmatrix}* & * & *\\
* & * & *\\
* & * & *
\end{smallmatrix}\right)t=\left(E_{3}\begin{smallmatrix}* & * & *\\
* & * & *\\
* & * & *
\end{smallmatrix}\right)\right\} \big/\sim,
\]
where $(\lambda g,\lambda^{-1}t)\sim(g,t)$ with $\lambda\in\mathbb{C}^{*}$.
It is easy to see that $T\simeq(\mC^{*})^{5}$. We denote by $M_{3,6}^{E_{3}}$
the subset of $M_{3,6}^{o}$ consisting of matrices of the form (\ref{eq:E3-gaugeA}).
We regard $M_{3,6}^{E_{3}}$ as a subset of the 9-dimensional affine
$\mC$-space $\mathbb{A}^{9}=\mathbb{C}^{9}$. Note that $M_{3,6}^{E_{3}}$
is an open dense subset in $\mathbb{A}^{9}$, and the $T$ action
naturally extends to $\mathbb{A}^{9}$. It is easy to read off the
weights of the $T\simeq(\mC^{*})^{5}$ actions on $\mathbb{A}^{9}$.
To do that we fix the isomorphism $T\simeq(\mC^{*})^{5}$, and present
the weights of the $T$-actions on $(\bm{a}\,\bm{b}\:\bm{c})\in\mathbb{A}^{9}$
in the following table: 
\begin{equation}
\begin{matrix}a_{0} & a_{1} & a_{2} & b_{0} & b_{1} & b_{2} & c_{0} & c_{1} & c_{2}\\
\hline 1 & 1 & 1 & 0 & 0 & 0 & 0 & 0 & 0\\
0 & 0 & 0 & 1 & 1 & 1 & 0 & 0 & 0\\
0 & 0 & 0 & 0 & 0 & 0 & 1 & 1 & 1\\
0 & -1 & -1 & 0 & 0 & 1 & 0 & 1 & 0\\
0 & 1 & 0 & 0 & -1 & -1 & 0 & 0 & 1
\end{matrix}\label{eq:tableWeights}
\end{equation}
The toroidal compactification $\mathcal{M}_{SecP}$ of $P(3,6)$ will
turn out to be a toric variety compactifying the quotient $\mathbb{A}^{9}/T$. 

\para{Toroidal compactification via the secondary fan} \label{para:SecPolytope}
As it will become clear when we describe the differential equations
of period integrals, the toric variety of the quotient $\mathbb{A}^{9}/T$
is given by the data of nine integral vectors which we read from the
nine column integral vectors in the table (\ref{eq:tableWeights}).
Following the convention in \cite{GKZ1}, reordering the columns slightly,
we define a finite set $\mathcal{A}$ of the integral vectors by 
\begin{equation}
\mathcal{A}:=\left\{ \left(\begin{smallmatrix}1\\
0\\
0\\
0\\
0
\end{smallmatrix}\right),\left(\begin{smallmatrix}0\\
1\\
0\\
0\\
0
\end{smallmatrix}\right),\left(\begin{smallmatrix}0\\
0\\
1\\
0\\
0
\end{smallmatrix}\right),\left(\begin{smallmatrix}1\\
0\\
0\\
-1\\
1
\end{smallmatrix}\right),\left(\begin{smallmatrix}1\\
0\\
0\\
-1\\
0
\end{smallmatrix}\right),\left(\begin{smallmatrix}0\\
1\\
0\\
0\\
-1
\end{smallmatrix}\right),\left(\begin{smallmatrix}0\\
1\\
0\\
1\\
-1
\end{smallmatrix}\right),\left(\begin{smallmatrix}0\\
0\\
1\\
1\\
0
\end{smallmatrix}\right),\left(\begin{smallmatrix}0\\
0\\
1\\
0\\
1
\end{smallmatrix}\right)\right\} .\label{eq:defA}
\end{equation}
 This set $\mathcal{A}$ is a finite set in $\mZ^{5}\equiv N$. We
denote by $M$ the dual of $N$ with the dual pairing $\langle\,,\,\rangle:M\times N\to\mZ$. 
\begin{prop}
The cone $Cone(\mathcal{A})$ generated by $\mathcal{A}$ is a Gorenstein
cone in $N_{\mathbb{R}}$, and satisfies 
\[
Cone(\mathcal{A})\cap\left\{ x\mid\langle m,x\rangle=1\right\} =\mathcal{A}
\]
 with $m=(1,1,1,0,0).$ \end{prop}
\begin{proof}
This can be verified by direct computations.
\end{proof}
We consider the regular triangulations of the convex hull $Conv(\mathcal{A})$.
Following \cite{GKZ2}, we have the so-called secondary polytope of
$\mathcal{A}$, which we denote by $Sec(\mathcal{A})$. See Appendix
\ref{sec:Appendix-SecA}. The secondary polytope is a lattice polytope
in $L_{\mathbb{R}}:=L\otimes\mR$ with 
\begin{equation}
L:=\mathrm{Ker}\left\{ \varphi_{\mathcal{A}}:\mZ^{\mathcal{A}}\to\mZ^{5}\right\} ,\label{eq:lattice-L}
\end{equation}
where $\varphi_{\mathcal{A}}$ is the integral linear map defined
by the $5\times9$ matrix obtained from $\mathcal{A}$ in (\ref{eq:defA}).
The normal fan of $Sec(\mathcal{A})$, called secondary fan, will
be denoted by $Sec\,\Sigma(\mathcal{A})$. The projective toric variety
$\mathbb{P}_{Sec(\mathcal{A})}$ for the polytope $Sec(\mathcal{A})$
in $L_{\mathbb{R}}=L\otimes\mathbb{R}$ is the toric variety giving
a natural compactification of the quotient $\mathbb{A}^{9}/T$. We
shall denote this compactification by $\mathcal{M}_{SecP}$.
\begin{prop}
\label{prop:SecP}The secondary polytope $Sec(\mathcal{A})\subset L_{\mathbb{R}}$
has 108 vertices. Except for six vertices, the cones from the vertices
are regular cones which define smooth affine charts (coordinate rings)
of $\mathcal{M}_{SecP}$. The affine charts corresponding to the 6
vertices are singular at the origin and are isomorphic to 
\[
\mathcal{M}_{SecP}^{Loc}=\mathrm{Spec}\,\mathbb{C}[C_{NE}\cap L]\simeq\left\{ (x_{ij})\in\mathbb{A}^{2\times3}\mid\mathrm{rk}\left(\begin{matrix}x_{11} & x_{12} & x_{13}\\
x_{21} & x_{22} & x_{23}
\end{matrix}\right)\leq1\right\} ,
\]
where $C_{NE}\subset L_{\mR}$ is the cone defined by 
\[
C_{NE}=\mathrm{Cone}\left\{ \begin{aligned}(-1,\;\;\,0,\;\;\,0,1,0,0,0,1,-1),\\
(\;\;\,0,-1,\,\;\;0,1,-1,1,0,0,0),\\
(\;\;\,0,\;\;\,0,-1,0,0,1,-1,1,0),
\end{aligned}
\;\begin{aligned}(-1,\;\;\,0,\;\;\,0,0,1,-1,1,0,0)\\
(\;\;\,0,-1,\,\;\;0,0,0,0,1,-1,1)\\
(\;\;\,0,\;\;\,0,-1,-1,1,0,0,0,1)
\end{aligned}
\right\} .
\]
\end{prop}
\begin{proof}
We can verify the claimed properties directly calculating the secondary
polytope. The cone $C_{NE}$ is described in Appendix \ref{sec:Appendix-SecA}. \end{proof}
\begin{rem}
One can also find more details about the combinatorics of the secondary
fan in \cite{Sekiguchi-Taka} .
\end{rem}
In the next section, we will observe that the convex hull of the 6
vertices coincides with a polytope which gives $\mathcal{M}_{3,3}$,
and that $\mathcal{M}_{SecP}$ gives a partial resolution of the singularities
of $\mathcal{M}_{3,3}$. This observation is the starting point of
our analysis of $E(3,6)$ defined on $\mathcal{M}_{6}$. 

Explicit forms of hypergeometric series of type $E(3,6:\alpha_{1},...,\alpha_{6})$
for general exponents $\alpha_{i}$ are considered in \cite{Sekiguchi-Taka}
by studying the combinatorial aspect of the secondary polytope $Sec(\mathcal{A})$.
However, it should be noted that our system $E(3,6)$ has special
values of exponents $\alpha_{1}=\cdots=\alpha_{6}=\frac{1}{2}$, which
belongs to the cases called \textit{resonant}, and is beyond the consideration
in \cite{Sekiguchi-Taka}. In fact, we need to find out detailed relationships
between the moduli spaces $\mathcal{M}_{6}$, $\mathcal{M}_{3,3}$
and $\mathcal{M}_{SecP}$ to write the solutions for this case. After
formulating the relationships, we will observe in Section \ref{sec:HyperG-D-Grassmann}
that the techniques in \cite{HKTY,HLY} developed in mirror symmetry
and the results in \cite{YoshidaEtal,Matsu93} merge quite nicely
in a general framework, i.e., \textit{$\mathcal{D}$-module on Grassmannians
\cite{Ao,Ao-Kita,GelfandGelfand}.} 

~

~

\section{\textbf{\textup{\label{sec:GKZ-from-E36}GKZ hypergeometric system
from $E(3,6)$}}}

It is known in general that the Aomoto-Gel'fand system on Grassmannians
is expressed by the Gel'fand-Geraev and Gel'fand-Kapranov-Zelevinski
system (GKZ system for short) when we reduce the $GL(n,\mathbb{C})$-action
to tori by making a ``partial gauge'' of the form (\ref{eq:E3-gaugeA})
(see \cite[Sect.3.3.4]{Ao-Kita}). Here we study the period integral
(\ref{eq:peridIntI}) with the reduced form (\ref{eq:E3-gaugeA})
to set up the GKZ system.

\subsection{GKZ hypergeometric system from $E(3,6)$}

Let us take the parameters in the six lines $\ell_{i}$ as in (\ref{eq:E3-gaugeA}).
Then we can write the holomorphic two form as 

\begin{equation}
\begin{aligned}\frac{d\mu}{\sqrt{\Pi_{i=1}^{6}\ell_{i}}} & =\frac{d\tx\wedge d\ty}{\sqrt{\tx\ty(a_{2}+a_{0}\tx+a_{1}\ty)(b_{1}+b_{2}\tx+b_{0}\ty)(c_{0}+c_{1}\tx+c_{2}\ty)}}\\
 & =\frac{1}{\sqrt{\big(a_{0}+\frac{a_{2}}{\tx}+a_{1}\frac{\ty}{\tx}\big)\big(b_{0}+\frac{b_{1}}{\ty}+b_{2}\frac{\tx}{\ty}\big)\big(c_{0}+c_{1}\tx+c_{2}\ty\big)}}\frac{d\tx\wedge d\ty}{\tx\ty},
\end{aligned}
\label{eq:1-over-f1f2f3}
\end{equation}
where we take the affine coordinate $\tz=1$ of $\mathbb{P}^{2}$.
We observe that the finite set $\mathcal{A}$ in (\ref{eq:defA})
can be interpreted as the exponents of the three Laurent polynomial
factors in the denominator, if we write $\mathcal{A}$ as follows:
\[
\begin{aligned}\mathcal{A}= & \bigg\{ e_{1}\times\left(\begin{matrix}0\\
0
\end{matrix}\right),e_{2}\times\left(\begin{matrix}0\\
0
\end{matrix}\right),e_{3}\times\left(\begin{matrix}0\\
0
\end{matrix}\right),\qquad\qquad\qquad\qquad\qquad\qquad\qquad\qquad\\
 & \quad e_{1}\times\left(\begin{matrix}-1\\
1
\end{matrix}\right),e_{1}\times\left(\begin{matrix}-1\\
0
\end{matrix}\right),e_{2}\times\left(\begin{matrix}0\\
-1
\end{matrix}\right),e_{2}\times\left(\begin{matrix}1\\
-1
\end{matrix}\right),e_{3}\times\left(\begin{matrix}1\\
0
\end{matrix}\right),e_{3}\times\left(\begin{matrix}0\\
1
\end{matrix}\right)\bigg\},
\end{aligned}
\]
where we $e_{1},e_{2},e_{3}$ are the basis of the first factor in
$\mZ^{5}=\mZ^{3}\times\mZ^{2}$. Let us write the three Laurent polynomial
factors as $f_{1}(a,\tx,\ty),f_{2}(b,\tx,\ty),f_{3}(c,\tx,\ty)$ so
that (\ref{eq:1-over-f1f2f3}) becomes 
\[
\frac{1}{\sqrt{f_{1}(a,\mathtt{x},\mathtt{y})f_{2}(b,\mathtt{x},\mathtt{y})f_{3}(c,\mathtt{x},\mathtt{y})}}\frac{d\mathtt{x}\wedge d\mathtt{y}}{\mathtt{xy}}.
\]
Observe the striking similarity with the corresponding forms we encountered
in a folklore paper \cite{HKTY}, except the appearance of the square
root in the denominator. 
\begin{prop}
\label{prop:GKZ-sys-def} Let $\mathcal{A}$ be as given in (\ref{eq:defA}).
The period integral (\ref{eq:peridIntI}) with its integrand (\ref{eq:1-over-f1f2f3})
satisfies GKZ $\mathcal{A}$-hypergeometric system \cite{GKZ1} with
exponents $\beta=\,^{t}(-\frac{1}{2},-\frac{1}{2},-\frac{1}{2},0,0).$\end{prop}
\begin{proof}
This follows easily by looking at invariance properties under the
torus action $T$ of the period integral, see \cite{HKTY,HLY}. The
only difference from there is in the exponent $\beta$, which is explained
by the square root in the denominator. We leave the derivation as
an easy exercise for the reader. \end{proof}
\begin{rem}
\label{rem:six-choices}From the first line to the second line of
(\ref{eq:1-over-f1f2f3}), the division by $\tx\ty$ has been made
by making a choice which factor of $\tx\ty$ goes to which factor
of the three parentheses. There are six combinatorially different
ways in total. Recall that we have chosen the isomorphism $T\simeq(\mC^{*})^{5}$
for the weights (\ref{eq:tableWeights}) so that the resulting set
$\mathcal{A}$ is compatible with the choice made in (\ref{eq:1-over-f1f2f3}).
We will return to this point in the next subsection. 
\end{rem}

\subsection{Boundary points (LCSLs) of the GKZ system\label{sub:GKZ-LCSLs}}

A fundamental object in mirror symmetry is a special boundary point
in the moduli space of Calabi-Yau manifolds, called a LCSL, which
appears as the intersection of certain normal crossing boundary divisors
of suitable compactification of the moduli space. In the case of Calabi-Yau
complete intersections in toric varieties, it is well known that such
compactifications are naturally obtained by finding a suitable toric
resolution of the compactification $\mathcal{M}_{SecP}$ \cite{HKTY,HLY}. 

\para{Resolutions of ${\mathcal M}_{SecP}$} \label{para:C_NE} Under
the identification $L\equiv\mZ^{4}$ in (\ref{eq:Appendix-Pi4}),
we have 
\begin{equation}
C_{NE}={\rm Cone}\left\{ \begin{matrix}(1,0,0,0),(0,1,-1,1),(1,-1,1,0),(0,0,0,1)\\
(0,0,1,-1),(-1,1,0,0)
\end{matrix}\right\} ,\label{eq:C_NE-4dim}
\end{equation}
for the cone $C_{NE}\subset L_{\mathbb{R}}$. 
\begin{lem}
\label{lem:C_NE-Resolutions} (1) The dual cone $C_{NE}^{\vee}$ is
generated by $\rho_{1},\cdots,\rho_{5}$ where 
\[
\begin{aligned} & \rho_{1}=(1,1,0,0), &  & \rho_{2}=(0,1,1,0), &  & \rho_{3}=(0,0,1,1),\\
 & \rho_{4}=(1,1,1,0), &  & \rho_{5}=(0,1,1,1).
\end{aligned}
\]
 (2) Without adding extra ray generators, there are two possible decompositions
of $C_{NE}^{\vee}$, namely,
\begin{equation}
C_{NE}^{\vee}=\sigma_{1}^{(1)}\cup\sigma_{2}^{(1)}=\sigma_{1}^{(2)}\cup\sigma_{2}^{(2)}\cup\sigma_{3}^{(2)}\label{eq:decomp-CNE-dual}
\end{equation}
with 
\[
\begin{aligned} & \sigma_{i}^{(1)}={\rm Cone}\left\{ \rho_{1},\rho_{2},\rho_{3},\rho_{6-i}\right\} \;(i=1,2)\text{ and }\\
 & \sigma_{i}^{(2)}{\rm =Cone}\left\{ \rho_{j},\rho_{k},\rho_{4},\rho_{5}\right\} \;\;(\{i,j,k\}=\{1,2,3\}).
\end{aligned}
\]
(3) All $\sigma_{i}^{(k)}$ are smooth simplicial cones, and hence
each in (2) defines a resolution of the singularity at the origin
of $\mathrm{Spec}\,\mathbb{C}[C_{NE}\cap L]$. The first and the second
decompositions in (2) correspond, respectively, to the left and the
right resolutions shown Fig.1. \end{lem}
\begin{proof}
All the claims can be verified by explicit calculations.\end{proof}
\begin{prop}
\label{prop:MsecP-Resolution}Choose a subdivision of (\ref{eq:decomp-CNE-dual}),
independently, at each of the six affine charts of $\mathcal{M}_{SecP}$
corresponding to the six singular vertices in Proposition \ref{prop:SecP}.
For each choice of the subdivisions, we have a resolution of $\mathcal{M}_{SecP}$,
and the difference of the choice in (\ref{eq:decomp-CNE-dual}) is
represented by four dimensional flip shown in Fig. 1. \end{prop}
\begin{proof}
Our proof is based on the explicit construction of the secondary fan
$Sec\,\Sigma(\mathcal{A})$, which consists of 108 four dimensional
cones. Since all cones except the six are smooth, we obtain a resolution
by choosing a subdivision for each of the six cones as claimed. The
four dimensional flip should be clear in the form of the singularity
expressed by the rank condition in Proposition \ref{prop:SecP}. 
\end{proof}
We shall write $\widetilde{\mathcal{M}}_{SecP}$ and $\widetilde{\mathcal{M}}_{SecP}^{+}$,
respectively, for the resolution where all six local resolutions are
of the left type and the right type in Fig.1. 

\begin{figure}[h] 
\resizebox{9.5cm}{!}{
\includegraphics[trim=3cm 4cm 0cm 1.5cm,
clip=true]{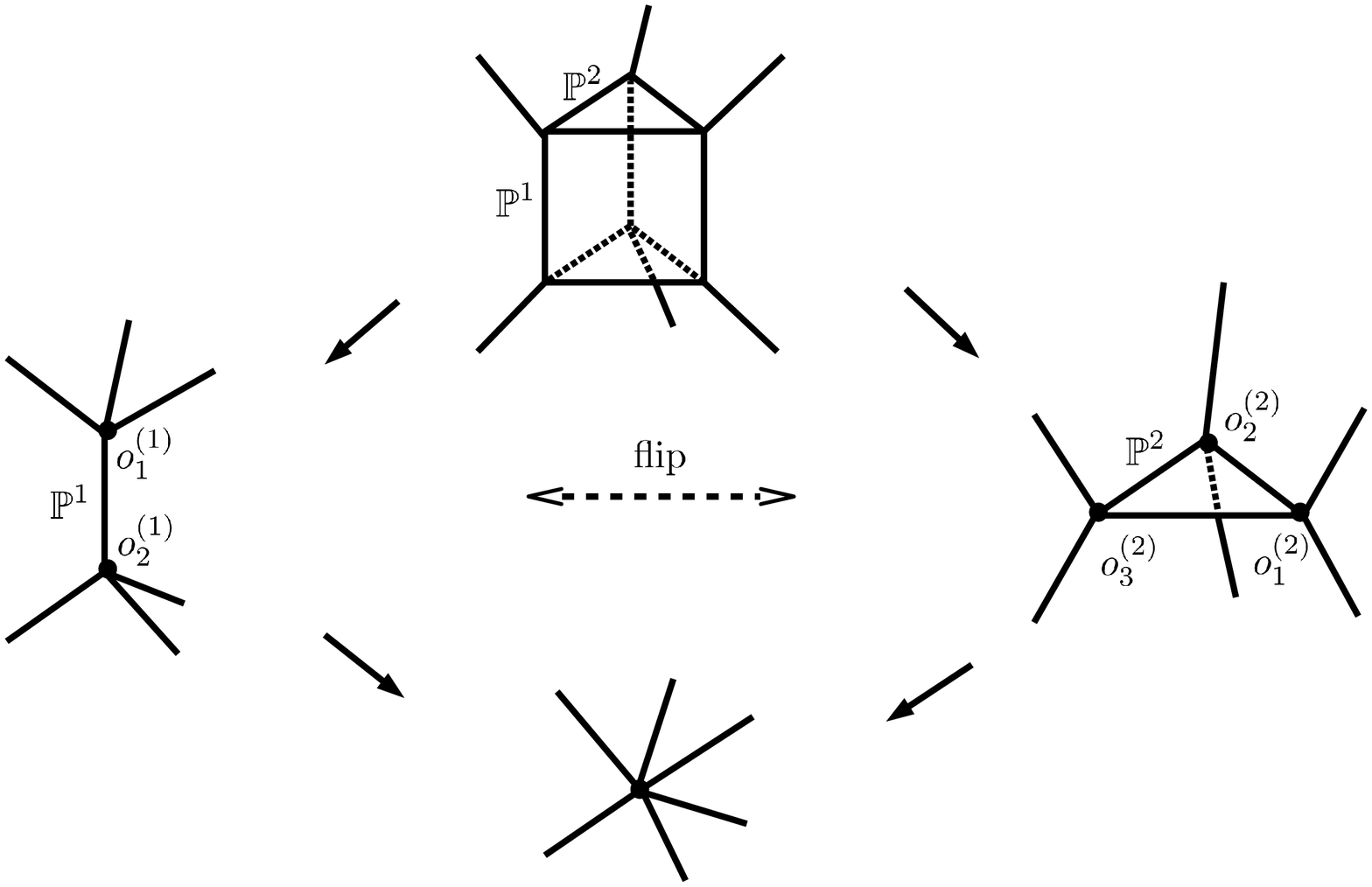}}
\caption{{\bf Fig.1} Four dimensional flip in the resolutions of $\mathrm{Spec}\,{\mathbb C}[C_{NE}\cap L]$. All boundary points $o_i^{(a)}$ are LCSLs. }
\end{figure}

~

\para{Power series solutions and Picard-Fuchs equations} In this
subsection, we give the power series solutions of the GKZ $\mathcal{A}$-hypergeometric
system near the LCSL in the the affine chart $\mathrm{Spec}\,\mathbb{C}[C_{NE}\cap L]$
(see Appendix \ref{sec:Appendix-SecA}). To simplify the form of the
power series, we normalize the period integral (\ref{eq:peridIntI})
as follows: 
\begin{equation}
\omega_{C}(a):=\sqrt{a_{0}b_{0}c_{0}}\;\bar{\omega}_{C}(a).\label{eq:wC-def}
\end{equation}

\begin{defn}
Let $(\sigma_{i}^{(k)})^{\vee}$ be the dual cone of $\sigma_{i}^{(k)}$
in (\ref{eq:decomp-CNE-dual}), which is smooth. We represent $(\sigma_{i}^{(k)})^{\vee}$
in $L_{\mathbb{R}}$ by using (\ref{eq:Appendix-Pi4}). Then in terms
of its primitive generators, we have 
\[
(\sigma_{i}^{(k)})^{\vee}=\mathrm{Cone}\left\{ \ell^{(1)},\ell^{(2)},\ell^{(3)},\ell^{(4)}\right\} .
\]
Let $z_{m}:=\mathtt{a^{\ell^{(m)}}}=\prod_{i=1}^{9}\mathtt{a}_{i}^{\ell_{i}^{(m)}}$
be the affine coordinates on $\mathrm{Spec}\,\mathbb{C}[(\sigma_{i}^{(m)})^{\vee}\cap L]$
with arranging the parameters 
\[
\mathtt{a}:=(-a_{0,}-b_{0},-c_{0},a_{1},a_{2},b_{1},b_{2},c_{1},c_{2}).
\]
Then the hypergeometric series associated to $\sigma_{i}^{(a)}$ is
defined to be 
\begin{equation}
\omega_{0}(z)=\sum_{n_{1},...,n_{4}\geq0}\frac{1}{\Gamma(\frac{1}{2})^{3}}\frac{\prod_{i=1}^{3}\Gamma(n\cdot\ell_{i}+\frac{1}{2})}{\prod_{i=4}^{9}\Gamma(n\cdot\ell_{i}+1)}z_{1}^{n_{1}}z_{2}^{n_{2}}z_{3}^{n_{3}}z_{4}^{n_{4}},\label{eq:w0-z}
\end{equation}
where $n\cdot\ell:=\sum_{k}n_{k}\ell^{(k)}$ (see \cite{HKTY,HLY}). 

~

The hypergeometric series $w_{0}(z)$ is the unique power series solution
of the GKZ $\mathcal{A}$-hypergeometric system on $\mathcal{M}_{SecP}$
near a LCSL point. We now use the method developed in \cite{HLY}
to determine the complete set of the Picard-Fuchs differential operators.
To show the calculations, we take the affine chart $\mathrm{Spec}\,\mathbb{C}[(\sigma_{1}^{(1)})^{\vee}\cap L]$
as an example. It should be clear that the constructions below are
parallel for the other cases $\mathrm{Spec}\,\mathbb{C}[(\sigma_{i}^{(k)})^{\vee}\cap L]$. 

As the primitive generator of $(\sigma_{1}^{(1)})^{\vee}\subset L_{\mR}$,
we first obtain 
\[
\begin{matrix}\ell^{(1)}=(-1,\;\;\,0,\;\;\,0,\;\;\,1,\;\;\,0,\;\;\,0,\;\;\,0,\;\;\,1,-1),\\
\ell^{(2)}=(\;\;\,0,-1,\;\;\,0,\;\;\,1,-1,\;\;\,1,\;\;\,0,\;\;\,0,\;\;\,0),\\
\ell^{(3)}=(\;\;\,0,\;\;\,0,-1,\;\;\,0,\;\;\,0,\;\;\,1,-1,\;\;\,1,\;\;\,0),\\
\ell^{(4)}=(\;\;\,0,\;\;\,0,\;\;\,0,-1,\;\;\,1,-1,\;\;\,1,-1,\;\;\,1).
\end{matrix}
\]
The power series (\ref{eq:w0-z}) now becomes 
\begin{equation}
\omega_{0}(z)=\sum_{n_{1},n_{2},n_{3},n_{4}\geq0}c(n_{1},n_{2},n_{3},n_{4})z_{1}^{n_{1}}z_{2}^{n_{2}}z_{3}^{n_{3}}z_{4}^{n_{4}}\label{eq:w0-local-sol}
\end{equation}
with the coefficients $c(n)=c(n_{1},n_{2},n_{3},n_{4})$ given by
\[
c(n):=\frac{1}{\Gamma(\frac{1}{2})^{3}}\frac{\Gamma(n_{1}+\frac{1}{2})\Gamma(n_{2}+\frac{1}{2})\Gamma(n_{3}+\frac{1}{2})}{\Pi_{_{i=1}}^{3}\Gamma(n_{4}-n_{i}+1)\cdot\Pi_{1\leq j<k\leq3}\Gamma(n_{j}+n_{k}-n_{4}+1)}.
\]
Picard-Fuchs differential equations may be characterized by the set
of differential operators which annihilate the power series $\omega_{0}(z)$.
In the present case, since the period integrals (normalized by $\sqrt{a_{0}b_{0}c_{0}}$)
satisfy the GKZ $\mathcal{A}$-hypergeometric system we can construct
them from the elements $\ell\in(\sigma_{1}^{(1)})^{\vee}\cap L$ .
The method in \cite{HLY} produces finite set of operators in terms
of Gr\"obner basis. 

Let $\ell=\ell_{+}-\ell_{-}$ be the unique decomposition under the
conditions $\ell_{\pm}\in\mZ_{\geq0}^{9}$ and ${\rm supp}(\ell_{+})\cap{\rm supp}(\ell_{-})=\phi$.
For such decomposition $\ell=\ell_{+}-\ell_{-}$, we define the GKZ
differential operator by $\square_{\ell}=\left(\frac{\partial\;}{\partial\mathtt{a}}\right)^{\ell_{+}}-\left(\frac{\partial\;}{\partial\mathtt{a}}\right)^{\ell_{-}}$.
We use the multi-degree convention $\mathtt{a}^{m}:=\Pi_{i=1}^{9}\mathtt{a}_{i}^{m_{i}}$
as above, and similarly for $\left(\frac{\partial\;}{\partial\mathtt{a}}\right)^{m}$.
Following the reference \cite{HLY}, we define 
\[
\mathtt{a}^{\ell_{+}}\square_{l}=\mathtt{a}^{\ell_{+}}\left(\frac{\partial\;}{\partial\mathtt{a}}\right)^{\ell_{+}}-\mathtt{a}^{\ell}\cdot\mathtt{a}^{\ell_{-}}\left(\frac{\partial\;}{\partial\mathtt{a}}\right)^{\ell_{-}},
\]
which we can express in terms of $\theta_{\mathtt{a}_{i}}:=\mathtt{a}_{i}\frac{\partial\;}{\partial\mathtt{a}_{i}}$
and a monomial of $z_{m}:=\mathtt{a}^{\ell^{(m)}}$ since $\ell\in(\sigma_{1}^{(1)})^{\vee}\cap L$
and $\ell^{(m)}$'s generate the cone. Our period integrals $\omega_{C}(a)$
are related to GKZ hypergeometric series by the factor $\sqrt{a_{0}b_{0}c_{0}}$
as in (\ref{eq:wC-def}), hence the differential operators 
\[
\mathcal{D}_{\ell}:=(a_{0}b_{0}c_{0})^{\frac{1}{2}}\left(\mathtt{a}^{\ell_{+}}\square_{l}\right)(a_{0}b_{0}c_{0})^{-\frac{1}{2}}
\]
annihilate the normalized period integrals $\omega_{C}(a).$ In Appendix
\ref{sec:Appendix-PF-eqs}, we list a minimal set of differential
operators which determine the period integrals around the origin of
the affine chart $\mathrm{Spec}\,\mathbb{C}[(\sigma_{1}^{(1)})^{\vee}\cap L]$. \end{defn}
\begin{prop}
\label{prop:LCSL-GKZ}The period integral $\omega_{0}(z)$ in (\ref{eq:w0-local-sol})
is the only power series solution near a LCSL given by the origin
of $\mathrm{Spec}\,\mathbb{C}[(\sigma_{1}^{(1)})^{\vee}\cap L]\simeq\mathbb{C}^{4}$.
The origin is the special point (LCSL) where all other linearly independent
solutions contain some powers of $\log z_{i}\;(i=1,...,4).$ \end{prop}
\begin{proof}
The first claim can be verified by the set of differential operators
in Appendix \ref{sec:Appendix-PF-eqs}. For the second claim, we will
find a closed formula for the logarithmic solutions. The closed formula
will be described in detail in \cite{HLTYpartII}. 
\end{proof}
Calculations are completely parallel for all other origins $o_{i}^{(k)}$
of the affine charts $\mathrm{Spec}\,\mathbb{C}[(\sigma_{i}^{(k)})^{\vee}\cap L]$
of the resolutions. One can verify the corresponding properties in
the above proposition hold for all $o_{i}^{(k)}$. 
\begin{rem}
\label{rem:six-choices-II} As noted in Remark \ref{rem:six-choices},
the six singular vertices in the secondary polytope $Sec(\mathcal{A})$
come from the combinatorial symmetry when reading $\mathcal{A}$ from
the period integral (\ref{eq:1-over-f1f2f3}). Hence, up to permutations
among the variables $a_{i}$, $b_{j}$ and $c_{k}$, respectively,
the hypergeometric series which we define for each of the six affine
chart have the same form as (\ref{eq:w0-local-sol}). Therefore the
Picard-Fuchs differential operators have the same form, up to suitable
conjugations by monomial factors, for all six affine charts of the
form $\mathrm{Spec}\,\mathbb{C}[C_{NE}\cap L]$ from the vertices
$T_{1},...,T_{6}$. Based on this simple property, we will have the
same Fourier expansions for the certain lambda functions when expanded
around the boundary points. Details are described in \cite{HLTYpartII}.
\end{rem}
~

~

\section{\textbf{\textup{\label{sec:M33-from-Periodint}$\mathcal{M}_{3,3}$
from period integrals }}}

As presented in \cite{HKTY,HLY} for the case of Calabi-Yau complete
intersections in toric varieties, GKZ hypergeometric systems provide
powerful means for calculating various predictions of mirror symmetry.
One may naively expects that this is also the case for $E(3,6)$.
However, it turns out that we need to further understand relationships
between the compactifications $\mathcal{M}_{SecP}$ , $\mathcal{M}_{3,3}$
and finally $\mathcal{M}_{6}$. In this section, we will find that
the compactification$\mathcal{M}_{3,3}$ arises naturally from evaluating
period integrals. We will see that $\mathcal{M}_{SecP}$ is actually
a partial resolution of $\mathcal{M}_{3,3}$.

\subsection{Power series from residue calculations}

Recall that, when determining Picard-Fuchs differential operators
in the previous section, we have normalized the period integral (\ref{eq:peridIntI})
by $\omega_{C}(a)=\sqrt{a_{0}b_{0}c_{0}}\;\bar{\omega}_{C}(a).$ Under
this normalization, by making use of the expansion $\frac{1}{\sqrt{1+P}}=\sum r_{n}P^{n}$,
we can evaluate the period integral over the torus cycle $\gamma=\left\{ \vert\tx\vert=\vert\ty\vert=\varepsilon\right\} $
as follows 
\[
\begin{aligned} & \begin{aligned}\int\frac{\sqrt{a_{0}b_{0}c_{0}}}{\sqrt{(a_{0}+\frac{a_{2}}{\tx}+a_{1}\frac{\ty}{\tx})(b_{0}+\frac{b_{1}}{\ty}+b_{2}\frac{\tx}{\ty})(c_{0}+c_{1}\tx+c_{2}\ty)}}\frac{d\tx d\ty}{\tx\ty}\end{aligned}
\\
 & =\int\sum_{n,m,k}r_{n}\left(\frac{a_{2}}{a_{0}}\frac{1}{\tx}+\frac{a_{1}}{a_{0}}\frac{\ty}{\tx}\right)^{n}r_{m}\left(\frac{b_{1}}{b_{0}}\frac{1}{\ty}+\frac{b_{2}}{b_{0}}\frac{\tx}{\ty}\right)^{m}r_{k}\left(\frac{c_{1}}{c_{0}}\tx+\frac{c_{2}}{c_{0}}\ty\right)^{k}\frac{d\tx d\ty}{\tx\ty}
\end{aligned}
\]
by formally evaluating the residues. 
\begin{lem}
\label{lem:w0-series}We have the period integrals over the torus
cycle $\gamma$ as a power series of 
\begin{equation}
x:=\frac{a_{2}c_{1}}{a_{0}c_{0}},y:=\frac{a_{1}b_{2}}{a_{0}b_{0}},z:=\frac{b_{1}c_{2}}{b_{0}c_{0}},u:=-\frac{a_{1}b_{1}c_{1}}{a_{0}b_{0}c_{0}},v:=-\frac{a_{2}b_{2}c_{2}}{a_{0}b_{0}c_{0}}\label{eq:vars-xyzuv}
\end{equation}
which satisfy the equation $xyz=uv$. Eliminating the powers of $v$,
the result is formally expressed by 
\begin{equation}
\omega_{0}(x,y,z,u):=\sum_{l=-\infty}^{\infty}\sum_{\;\;\;m,n,k\geq\max\left\{ 0,-l\right\} }c(n,m,k,l)x^{n}y^{m}z^{k}u^{l},\label{eq:w0-xyzu}
\end{equation}
where 
\[
c(n,m,k,l):=\frac{1}{\Gamma(\frac{1}{2})^{3}}\frac{\Gamma(m+n+l+\frac{1}{2})\Gamma(n+k+l+\frac{1}{2})\Gamma(m+k+l+\frac{1}{2})}{m!\,n!\,k!\,(m+l)!\,(n+l)!\,(k+l)!}.
\]
\end{lem}
\begin{proof}
The evaluation of the residues is straightforward (cf. \cite{Batyrev-Cox,HLY}).
The closed formula of the coefficients $c(n,m,k,l)$ can be deduced
from the formal solutions of the GKZ system \cite{GKZ1}. \end{proof}
\begin{prop}
\label{prop:w0-xyz-uv}The Laurent series $\omega_{0}(x,y,z,u)$ defines
a regular solutions around a point $[0,0,0,0,0,1]\in\mathcal{M}_{3,3}$
under the following identification of the parameters $x,y,z,u,v$
with the affine coordinate of $\mathcal{M}_{3,3}$:
\[
\left\{ [x,y,z,u,v,1]\in\mathbb{P}^{5}\mid xyz=1\,uv\right\} \subset\mathcal{M}_{3,3}.
\]
\end{prop}
\begin{proof}
By the definitions of $x,y,z,u,v$, we have the relation $xyz=uv$.
The claim is clear since we have a power series of $x,y,z,u,v$ in
Lemma \ref{lem:w0-series} (before eliminating $v$). \end{proof}
\begin{rem}
Recall that we have made a choice, among six combinatorial possibilities,
from the first line to the second line of (\ref{eq:1-over-f1f2f3})
as noted in Remark \ref{rem:six-choices}. It is easy to deduce that,
if we change our choice there, we will have the same power series
but with different variables, which corresponds to expansions around
different coordinate points of $\mathbb{P}^{5}$ (cf. Remark \ref{rem:six-choices-II}).
Namely, when we reduce the $GL(3,\mathbb{C})$ symmetry to the diagonal
tori as in (\ref{eq:E3-gaugeA}), we may consider that the period
integral (\ref{eq:1-over-f1f2f3}) is defined on 
\[
\mathcal{M}_{3,3}=\left\{ X_{0}X_{1}X_{2}=X_{3}X_{4}X_{5}\right\} \subset\mathbb{P}^{5}.
\]

\end{rem}
~

\subsection{\label{sub:M33Loc}$\mathcal{M}_{SecP}$ and $\mathcal{M}_{3,3}$ }

We have seen in Proposition \ref{prop:LCSL-GKZ} that the special
boundary points (LCSLs) appear in the resolutions of $\mathcal{M}_{SecP}$
. Here it turns out that $\mathcal{M}_{SecP}$ gives a partial resolution
of $\mathcal{M}_{3,3}$. 
\begin{prop}
\label{prop:lines-in-M33}The toric hypersurface $\mathcal{M}_{3,3}\subset\mathbb{P}^{5}$
contains all coordinate lines of $\mathbb{P}^{5}.$ The singularities
of $\mathcal{M}_{3,3}$ consist of six coordinate points $p_{i}(i=0,..,5)$
of $\mathbb{P}^{5}$ and nine coordinate lines $\overline{p_{i}p_{j}}$
$(0\leq i\leq2,3\leq j\leq5)$. \end{prop}
\begin{proof}
Since all claims are easy to verify from the defining equation of
the hypersurface, we omit the proofs.
\end{proof}
\begin{figure}[h] 
\resizebox{9cm}{!}{
\includegraphics[trim=0cm 4cm 0cm 4cm,
clip=true]{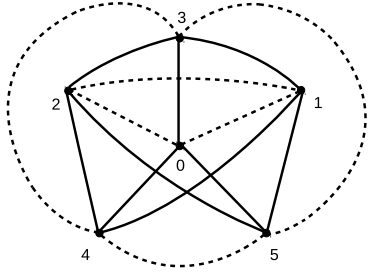}}
\caption{{\bf Fig.2} Singularities of ${\mathcal M}_{3,3}$. Solid lines represent the coordinate lines $\overline{p_ip_j}$ $(0\leq i\leq 2,3 \leq j \leq 5)$ along which ${\mathcal M}_{3,3}$ is singular. Broken lines are the other coordinate lines contained in ${\mathcal M}_{3,3}$.  }
\end{figure}

~

~

The following lemma is our first step to relate $\mathcal{M}_{SecP}$
and $\mathcal{M}_{3,3}$. To state it, we recall that the the secondary
polytope $Sec(\mathcal{A})$ has 108 vertices, whose associated cones
define coordinate rings of the affine charts of $\mathcal{M}_{SecP}$.
Of the 108 vertices, the six vertices $\mathcal{V}$ given in Appendix
\ref{sec:Appendix-SecA} are singular while the rest are smooth (see
Proposition \ref{prop:SecP}). 
\begin{lem}
We have $\mathcal{M}_{3,3}=\mathbb{P}_{Conv(\mathcal{V})}$. \end{lem}
\begin{proof}
This follows from the explicit calculation of $Conv(\mathcal{A})$.
We list the six vertices $\mathcal{V}$ of $Sec(\mathcal{A})$ in
Appendix \ref{sec:Appendix-SecA}. From the list, it is straightforward
to see the claim. 
\end{proof}
By the obvious symmetry of $\mathcal{M}_{3,3}$, we may restrict our
attention to the local affine geometry 
\[
\mathcal{M}_{3,3}^{Loc}:=\left\{ xyz=uv\right\} \subset\mC^{5},
\]
and deduce its relation to the resolution $\widetilde{\mathcal{M}}_{SecP}$.
If we read the exponents of the variables in (\ref{eq:vars-xyzuv}),
we can write the toric singularity $\mathcal{M}_{3,3}^{Loc}$ using
the lattice (\ref{eq:lattice-L}) as 
\[
\mathcal{M}_{3,3}^{Loc}=\mathrm{Spec}\,\mC[C_{0}\cap L],
\]
where 
\begin{equation}
C_{0}:=\mathrm{Cone}\left\{ \begin{aligned}(-1,\;\;\,0,-1,0,1,0,0,1,0),\\
(-1,-1,\,\;\;0,1,0,0,1,0,0),\\
(\;\;\,0,-1,-1,0,0,1,0,0,1),\\
(-1,-1,-1,1,0,1,0,1,0),\\
(-1,-1,-1,0,1,0,1,0,1)\,
\end{aligned}
\right\} .\label{eq:cone-C0}
\end{equation}
Note that the five generators $\ell\in L$ of $C_{0}$ listed here
express the the affine coordinates $x,y,z,u,v$ in (\ref{eq:vars-xyzuv})
by the monomials $\mathtt{a}^{\ell}$. Under (\ref{eq:Appendix-Pi4}),
we can also write $C_{0}$ by 
\[
C_{0}=\mathrm{Cone}\left\{ (0,1,0,0),(1,0,0,1),(0,0,1,0),(1,0,1,0),(0,1,0,1\}\right\} \subset\mathbb{R}^{4}.
\]
Note, from the form of $C_{0}$ in (\ref{eq:cone-C0}) and $C_{NE}$
in Appendix \ref{sec:Appendix-SecA}, that $C_{0}$ and $C_{NE}$
are cones from the same vertex $T_{1}$ of $Sec(\mathcal{A})$. 
\begin{lem}
\label{lem:C_NE-C_0} We have $C_{0}\subset C_{NE}\subset L_{\mathbb{R}}$
for the cone $C_{NE}$. \end{lem}
\begin{proof}
Since the vertex is chosen in common for $C_{NE}$ and $C_{0}$, the
claimed inclusion is easy to verify.
\end{proof}
In Appendix A, we have listed the primitive generators of the dual
cone $C_{0}^{\vee}$, which we denote by $\mu_{1},\cdots,\mu_{6}$
in order. Similarly we write the primitive generators of the dual
cone $C_{NE}^{\vee}$ by $\rho_{1},...,\rho_{5}.$ Note that, by Lemma
\ref{lem:C_NE-C_0}, we have the reversed inclusion \textit{as a set}
for the dual cones, i.e.,
\[
\mathrm{supp}\,C_{0}^{\vee}\supset\mathrm{supp\,}C_{NE}^{\vee}
\]
holds for the supports, in particular, the rays generated by $\rho_{1},\cdots,\rho_{5}$
are contained in $C_{0}^{\vee}$. Recall that the dual cone $C_{NE}^{\vee}$
has two possible subdivisions into smooth simplicial cones as described
in Lemma \ref{lem:C_NE-Resolutions} (2). In the following lemma,
we consider subdivisions of the dual cone $C_{0}^{\vee}$ using all
rays generated by $\mu_{1},\cdots,\mu_{6},\rho_{1},\cdots,\rho_{5}$. 
\begin{lem}
\label{lem:C_0-Dual-decomp} Up to the subdivisions of $C_{NE}^{\vee}$
in Lemma \ref{lem:C_NE-Resolutions} (2), there is a unique subdivision
of $C_{0}^{\vee}$ into smooth simplicial cones which contains the
dual cone $C_{NE}^{\vee}$ as a simplicial subset. \end{lem}
\begin{proof}
By explicit construction of all possible subdivisions, via a C++ code
TOPCOM \cite{TOPCOM}, we find 54 subdivisions. We verify the claimed
property from them.\end{proof}
\begin{lem}
\label{lem:partial-Res-M33loc}By the unique subdivision of $C_{0}^{\vee}$
in Lemma \ref{lem:C_0-Dual-decomp} which contains $C_{NE}^{\vee}$
as the simplicial subset, we have a partial resolution of the singularity
$\mathcal{M}_{3,3}^{Loc}=\mathrm{Spec}\,\mathbb{C}[C_{0}\cap L]$.\end{lem}
\begin{proof}
The claim is clear, since $C_{0}^{\vee}$ consists of smooth cones
up to subdivisions of $C_{NE}^{\vee}$.\end{proof}
\begin{prop}
\label{prop:Toric-Res-MsecP}The partial resolutions at each singular
points gives globally a partial resolution $\mathcal{M}_{SecP}\to\mathcal{M}_{3,3}.$ \end{prop}
\begin{proof}
Our proof is based on the explicit coordinate description of $\mathcal{M}_{SecP}$
calculating the secondary polytope. See also Remark \ref{rem:SecP-Resol-Loc}
below.\end{proof}
\begin{rem}
\label{rem:SecP-Resol-Loc} Toric resolutions of $\mathcal{M}_{SecP}$
have been described in Proposition \ref{prop:resolution-MsecP2}.
In the next section, we will obtain the same resolutions by blowing-up
along the singular locus of $\mathcal{M}_{3,3}$ (Proposition \ref{prop:resolution-MsecP2}).
In Fig.4, we depict one of the two possible resolutions of $\mathcal{M}_{3,3}^{Loc}$
schematically. As we see from the picture, the resolution of the singularity
is covered by 19 affine coordinate charts which correspond to 19 maximal
dimensional cones in the subdivision of $C_{0}^{\vee}$. If we remove
the subdivision of $C_{NE}^{\vee}\subset C_{0}^{\vee}$, then the
number reduces to 18, which is explained by 17 smooth maximal cones
and one singular cone $C_{NE}^{\vee}$ corresponding to $\mathrm{Spec}\,\mathbb{C}[C_{NE}\cap L]$.
One can also see the claim in Proposition \ref{prop:Toric-Res-MsecP}
in a simple counting $18\times6=108$ (see Proposition \ref{prop:SecP}).

~

~

~
\end{rem}

\section{\textbf{\textup{\label{sec:MoreOn-M33}More on the resolutions of
$\mathcal{M}_{3,3}$ }}}

In this section, we will describe the resolution without recourse
to the toric geometry of the secondary fan. This will allow us to
relate $\mathcal{M}_{SecP}$ to the geometry of the Baily-Borel-Satake
compactification $\mathcal{M}_{6}$. Recall that we have defined 
\[
\mathcal{M}_{3,3}^{Loc}=\mathrm{Spec}\,\mathbb{C}[C_{0}\cap L]\simeq\left\{ (x,y,z,u,v)\mid xyz=uv\right\} ,
\]
which which describes the local structure of the singularities in
$\mathcal{M}_{3,3}$. We shall write $\mathcal{X}=\mathcal{M}_{3,3}^{Loc}$
for short in what follows.

\subsection{\label{sub:blow-up-tildeX}Blowing-up $\mathcal{X}'\to\mathcal{X}$
along the singular locus}

From the defining equation $xyz=uv$, it is easy to see that the affine
hypersurface $\mathcal{X}\subset\mathbb{C}^{5}$ is singular along
the three coordinate lines of $x,y,z$ coordinates (cf. Subsection
\ref{sub:M33Loc}). Note that we can write the union of these lines
in $\mathbb{C}^{5}$ by 
\[
\Gamma:=\{u=v=xy=yz=zx=0\}.
\]
We will consider the blow-up $\pi_{1}\colon\mathcal{X}'\to\mathcal{X}$
along this locus $\Gamma$. Let us first introduce the blow up $\widetilde{\mathbb{C}_{\Gamma}^{5}}\subset\mathbb{C}^{5}\times\mathbb{P}^{4}$
starting with the relations 
\[
u:v:yz:zx:xy=U:V:W_{1}:W_{2}:W_{3},
\]
for $(u,v,x,y,z)\times[U,V,W_{1},W_{2},W_{3}]\in\mathbb{C}^{5}\times\mathbb{P}^{4}$.
The ideal $I_{\widetilde{\mathbb{C}_{\Gamma}^{5}}}$ of the blow-up
$\widetilde{\mathbb{C}_{\Gamma}^{5}}\subset\mathbb{C}^{5}\times\mathbb{P}^{4}$
is an irreducible component of the scheme defined by the above relations.
We denote by $\pi_{0}:\widetilde{\mathbb{C}_{\Gamma}^{5}}\rightarrow\mathbb{C}^{5}$
the natural projection. Then the blow-up $\mathcal{X}'$ is the strict
transform of $\mathcal{X}\subset\mathbb{C}^{5}$ by the birational
map $\pi_{0}$. 
\begin{prop}
\label{prop:Blow-Up-Xp}The blow-up $\mathcal{X}'$ is given in $\mathbb{C}^{5}\times\mathbb{P}^{4}$
by the following equations:
\begin{equation}
\begin{matrix}\begin{aligned}W_{1}W_{2} & =UVz, &  &  & W_{1}W_{3} & =UVy, &  &  & W_{2}W_{3} & =UVx,\\
W_{1}x & =Uv &  &  & W_{2}y & =Uv &  &  & W_{3}z & =Uv\\
W_{1}x & =Vu, &  &  & W_{2}y & =Vu, &  &  & W_{3}z & =Vu
\end{aligned}
\end{matrix}\label{eq:Ideal-Bl-Xp}
\end{equation}
and
\begin{equation}
\begin{matrix}\begin{aligned}W_{1}u & =yzU, &  &  & W_{2}u & =zxU, &  &  & W_{3}u & =xyU,\\
W_{1}v & =yzV, &  &  & W_{2}v & =zxV, &  &  & W_{3}v & =xyV.
\end{aligned}
\end{matrix}\label{eq:Ideal-Bl-Xp-C5}
\end{equation}
\end{prop}
\begin{proof}
The ideal $I_{\widetilde{\mathbb{C}_{\Gamma}^{4}}}$ and the equation
$xyz=uv$ define the ideal $I_{T}$ of the total transform of $\mathcal{X}$.
Calculating the primary decomposition of $I_{T}$ by Singular \cite{DGPS-Singular},
we see that the claimed equations generate the ideal of $\mathcal{X}'$. 
\end{proof}
\begin{figure}[h] 
\resizebox{9cm}{!}{
\includegraphics[trim=0cm 0.5cm 0cm 0.5cm,
clip=true]{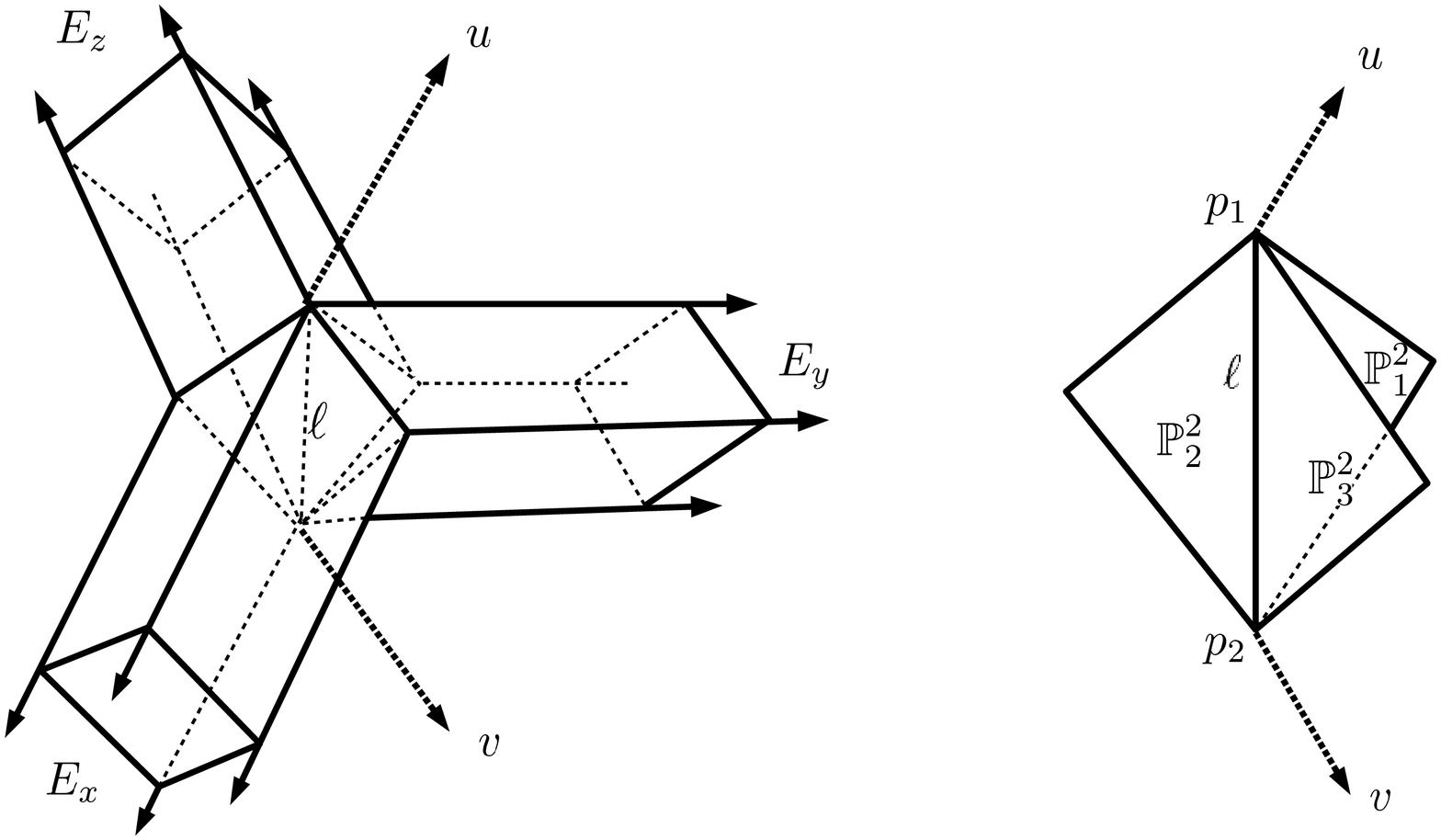}}
\caption{{\bf Fig.3} Exceptional divisors $E_x,E_y$ and $E_z$ in the blow-up $\mathcal{X}'$. Their junction locus is scaled up in the right figure. }
\end{figure}
\begin{prop}
\label{prop:BlowUp-along-xyz}The blow-up $\pi_{1}:\mathcal{X}'\to\mathcal{X}$
has the following properties:\end{prop}
\begin{enumerate}
\item The $\pi_{1}$-exceptional divisor has three irreducible components;
one for each coordinate line of $x,y,z$ coordinates. We call the
irreducible components $E_{x}$, $E_{y}$, $E_{z}$, respectively. 
\item The components $E_{x}$, $E_{y}$, $E_{z}$ have fibrations over the
corresponding coordinate lines. The $\pi_{1}$-fiber over a point
$p\in\Gamma$ is $(\mathbb{P}^{1})^{2}$ if $p$ is not the origin
$o$, while over the origin it is the union of three copies $\mathbb{P}_{i}^{2}\,(i=1,2,3)$
of $\mP^{2}$ which are glued along one line $\ell:=\mathbb{P}^{1}$
(see Fig.3). Over the origin, the components $E_{x},E_{y},E_{z}$
glue together by the following relations: 
\[
E_{x}|_{\pi_{1}^{-1}(o)}=\mathbb{P}_{2}^{2}\cup\mathbb{P}_{3}^{2},\;\;E_{y}|_{\pi_{1}^{-1}(o)}=\mathbb{P}_{3}^{2}\cup\mathbb{P}_{1}^{2},\;\;E_{z}|_{\pi_{1}^{-1}(o)}=\mathbb{P}_{1}^{2}\cup\mathbb{P}_{2}^{2}.
\]

\item The blow-up $\mathcal{X}'$ is singular only at two isolated points,
say, $p_{1}$ and $p_{2}$ on $\ell$. The singularities at these
points are isomorphic to the affine cone over the Segre$(\mP^{1})^{3}$.
\item The components $E_{x},E_{y}$ and $E_{z}$ are singular only at $p_{1}$
and $p_{2}$ with ODPs. \end{enumerate}
\begin{proof}
The claimed properties follow from the equations in Proposition \ref{prop:Blow-Up-Xp}.
For (1) and (2), because of the obvious symmetry, we only need to
consider the case of $x$-axes. Set $y=z=u=v=0$ in (\ref{eq:Ideal-Bl-Xp})
assuming $x\not=0$. Then we obtain $W_{1}=0$ and $W_{2}W_{3}=UVx$,
from which we see $\pi_{1}^{-1}(p)\simeq\mathbb{P}^{1}\times\mathbb{P}^{1}$
$(p\not=o)$ as claimed. When $x=y=z=u=v=0$, the equations (\ref{eq:Ideal-Bl-Xp})
become $W_{1}W_{2}=W_{1}W_{3}=W_{2}W_{3}=0$, from which we obtain
\[
\begin{aligned}\pi_{1}^{-1}(o) & =\left\{ o\times[U,V,W_{1},0,0]\right\} \cup\left\{ o\times[U,V,0,W_{2},0]\right\} \cup\left\{ o\times[U,V,0,0,W_{3}]\right\} \\
 & =:\mathbb{P}_{1}^{2}\cup\mathbb{P}_{2}^{2}\cup\mathbb{P}_{3}^{2}.
\end{aligned}
\]
Also we see that $\mathbb{P}_{1}^{2}\cap\mathbb{P}_{2}^{2}\cap\mathbb{P}_{3}^{2}=\left\{ o\times[U,V,0,0,0]\right\} =:\ell$
as claimed. It is easy to see the claimed forms of $E_{x}\vert_{\pi_{1}^{-1}(o)},E_{y}\vert_{\pi_{1}^{-1}(o)}$
and $E_{x}\vert_{\pi_{1}^{-1}(o)}$. 

To show (3), we express $\mathcal{X}'$ in affine coordinates. By
obvious symmetry, we only have to consider $\mathcal{X}'\vert_{W_{1}\not=0}$
and $\mathcal{X}'\vert_{U\not=0}$. Let us first describe the restriction
$\mathcal{X}'\vert_{W_{1}\not=0}$ by setting $W_{1}=1$. Then we
obtain the relations 
\[
W_{2}=UVz,\;\;W_{3}=UVy,\;\;x=Vu
\]
from (\ref{prop:Blow-Up-Xp}) and also $u=Uyz,\;\;v=Vyz$ from (\ref{eq:Ideal-Bl-Xp-C5}).
From these relations, we see that $\mathcal{X}'\vert_{W_{1}\not=0}$
is isomorphic to $\mathbb{C}^{4}$ with the coordinates $y,z,U,V$.
By symmetry, similar results hold for other cases $W_{2}\not=0$ and
$W_{3}\not=0$. In particular, $\mathcal{X}'\vert_{W_{i}\not=0}$
are smooth for $i=1,2,3.$ 

Next, let us describe $\mathcal{X}'\vert_{U\not=0}$ by setting $U=1$.
From (\ref{prop:Blow-Up-Xp}), we obtain 
\begin{equation}
\begin{matrix}\begin{aligned}W_{1}W_{2} & =Vz, &  &  & W_{1}W_{3} & =Vy, &  &  & W_{2}W_{3} & =Vx,\\
W_{1}x & =Vu, &  &  & W_{2}y & =Vu, &  &  & W_{3}z & =Vu
\end{aligned}
\end{matrix}\label{eq:Segre1}
\end{equation}
in addition to $v=Vu$ which eliminates $v$. Also from (\ref{eq:Ideal-Bl-Xp-C5}),
we have 
\begin{equation}
\begin{aligned}W_{1}u=yz, &  & W_{2}u=zx, &  & W_{3}u=xy\end{aligned}
\label{eq:Segre2}
\end{equation}
and also $W_{1}v=yzV,W_{2}v=zxV,W_{3}v=xyV$, where the latter three
relations are consequences other relations. We note that the equations
(\ref{eq:Segre1}) and (\ref{eq:Segre2}) are determinants of $2\times2$
sub-matrices of the 2-hypermatrix given in the equation below. Moreover,
the relations (\ref{eq:Segre1}) and (\ref{eq:Segre2}) are solved
by $a_{ijk}$ written in terms of the homogeneous coordinates $([a_{0},a_{1}],[b_{0},b_{1}],[c_{0},c_{1}])\in(\mathbb{P}^{1})^{3}$;

\def\hyperDet{
\begin{xy}
(-30,0)*++{a_0b_1c_1}="aoii",
(-10,0)*++{a_1 b_1 c_1}="aiii",
(-40,-10)*++{a_0b_0c_1}="aooi",
(-20,-10)*++{a_1b_0c_1}="aioi",
(-30,-20)*++{a_0b_1c_0}="aoio",
(-10,-20)*++{a_1b_1c_0}="aiio",
(-40,-30)*++{a_0b_0c_0}="aooo",
(-20,-30)*++{a_1b_0c_0}="aioo",
(-30,-10)*+{}="pi",
(-20,-20)*+{}="pii",
(-57,-15)*{\big(a_{ijk}\big)\;\;=},
(0,-15)*{=},
(20,0)*++{V}="Boii",
(40,0)*++{W_2}="Biii",
(10,-10)*++{W_1}="Booi",
(30,-10)*++{z}="Bioi",
(20,-20)*++{W_3}="Boio",
(40,-20)*++{x}="Biio",
(10,-30)*++{y}="Booo",
(30,-30)*++{u}="Bioo",
(20,-10)*+{}="qi",
(30,-20)*+{}="qii",
\ar@{-} "aoii";"aiii"
\ar@{-} "aooi";"aoii"
\ar@{-} "aooi";"aioi"
\ar@{-} "aioi";"aiii"
\ar@{-} "aoio";"pii"
\ar@{-} "pii";"aiio"
\ar@{-} "aooo";"aoio"
\ar@{-} "aooo";"aioo"
\ar@{-} "aioo";"aiio"
\ar@{-} "aoio";"pi"
\ar@{-} "pi";"aoii"
\ar@{-} "aooo";"aooi"
\ar@{-} "aioo";"aioi"
\ar@{-} "aiio";"aiii"
\ar@{-} "Boii";"Biii"
\ar@{-} "Booi";"Boii"
\ar@{-} "Booi";"Bioi"
\ar@{-} "Bioi";"Biii"
\ar@{-} "Boio";"qii"
\ar@{-} "qii";"Biio"
\ar@{-} "Booo";"Boio"
\ar@{-} "Booo";"Bioo"
\ar@{-} "Bioo";"Biio"
\ar@{-} "Boio";"qi"
\ar@{-} "qi";"Boii"
\ar@{-} "Booo";"Booi"
\ar@{-} "Bioo";"Bioi"
\ar@{-} "Biio";"Biii"
\end{xy} }
\[
\begin{matrix}\hyperDet\end{matrix}
\]
Thus we see that the relations (\ref{eq:Segre1}) and (\ref{eq:Segre2})
define the affine cone of the Segre$(\mathbb{P}^{1})^{3}$ in $\mathbb{C}^{8}$
with the affine coordinates $x,y,z,u,V,W_{1},W_{2},W_{3}$, which
is singular at the vertex (the origin of $\mathbb{C}^{8}$). Note
that the vertex corresponds to the point $p_{1}:=o\times[1,0,0,0,0]\in\mathcal{X}'$
which is on the line $\ell=\left\{ o\times[U,V,0,0,0]\right\} $.
By symmetry, the other case $\mathcal{X}'\vert_{V\not=0}$ can be
described similarly with the vertex $p_{2}:=o\times[0,1,0,0,0]$ on
the line $\ell$. 

The claim (4) follows from the proof for (3). For example, we set
$y=z=u=v=0$ in the equations (\ref{eq:Segre1}) and (\ref{eq:Segre2}).
Then we can verify the claimed property for $E_{x}$. 
\end{proof}
Note that $p_{1}$ and $p_{2}$ are the only singular points of $\mathcal{X}'$.
Let $\pi_{2}:\widetilde{\mathcal{X}}\to\mathcal{X}'$ be the blow-up
at $p_{1}$ and $p_{2}$. We denote by $\widetilde{E}_{x}$, $\widetilde{E}_{y}$,
$\widetilde{E}_{z}$ the strict transforms of the $\pi_{1}$-exceptional
divisors $E_{x}$, $E_{y}$, $E_{z}$ respectively. 
\begin{prop}
\label{prop:BlowUp-at-p1p2}The blow-up $\pi_{2}:\widetilde{\mathcal{X}}\to\mathcal{X}'$
introduces exceptional divisors $D_{p_{1}},D_{p_{2}}$ which are isomorphic
to $(\mP^{1})^{3}$. The resulting composite of the blow-ups of $\mathcal{X}$
gives a resolution of singularities $\pi_{1}\circ\pi_{2}:\widetilde{\mathcal{X}}\to\mathcal{X}$.
Moreover, the union $\widetilde{E}_{x}\cup\widetilde{E}_{y}\cup\widetilde{E}_{z}\cup D_{p_{1}}\cup D_{p_{2}}$
is a simple normal crossing divisor.\end{prop}
\begin{proof}
The first two claims follow from Proposition \ref{prop:BlowUp-along-xyz}.
The last assertion also follows from the explicit computations. 
\end{proof}
\begin{figure}[h] 
\resizebox{8.5cm}{!}{
\includegraphics[trim=0cm 1.5cm 0cm 2.5cm,
clip=true]{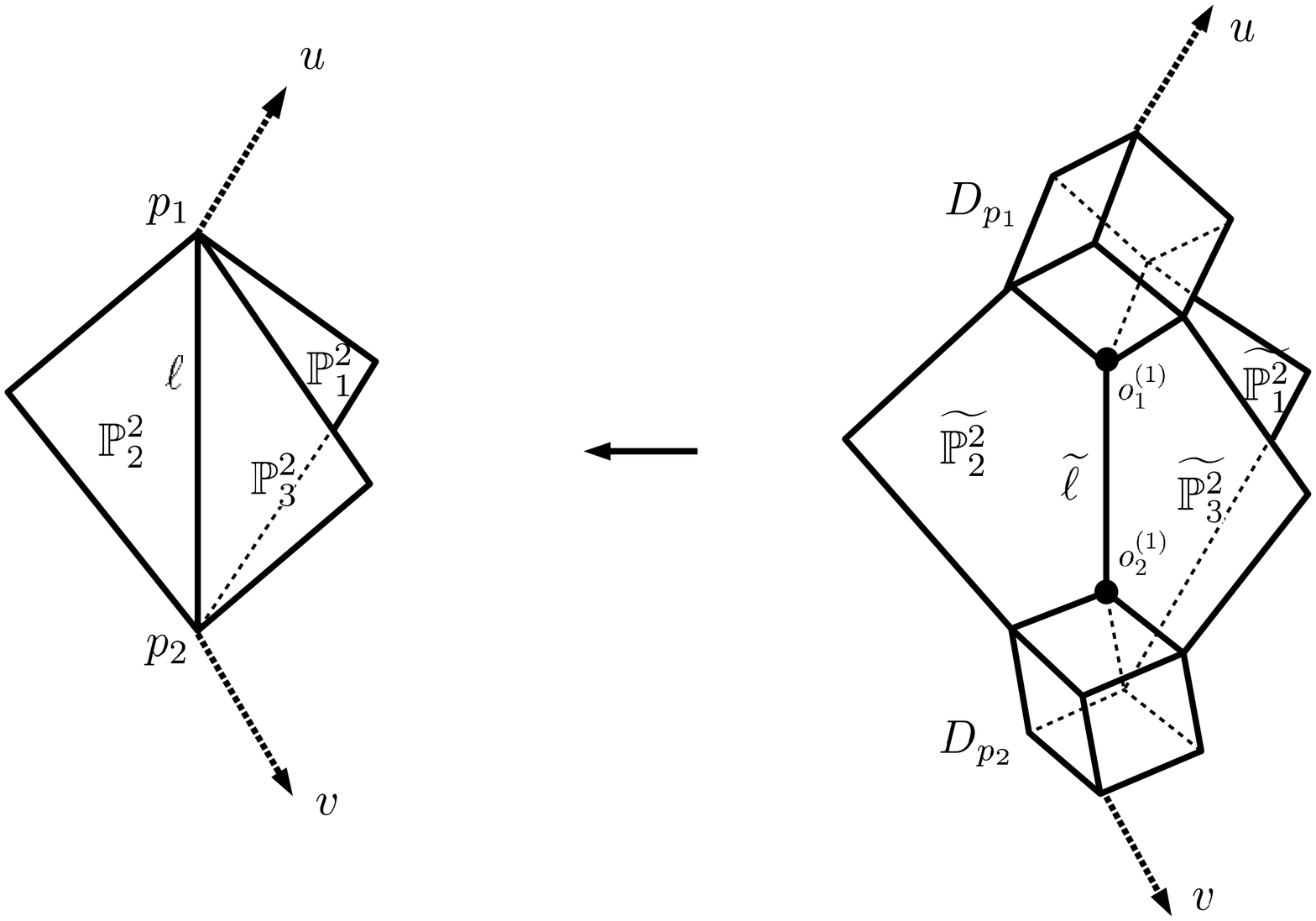}}
\caption{{\bf Fig.4} The blow-up of $\mathcal{X}'$ at $p_1$, $p_2$ in the junction. The intersection points $o^{(1)}_k=\widetilde{E_x}\cap\widetilde{E_y}\cap\widetilde{E_z}\cap D_{p_k}$ $(k=1,2)$ and $\widetilde{\ell}$ can be identified in the left of Fig.1. }
\end{figure}
\begin{rem}
\label{rem:def-Fi}As shown in Fig.4, the strict transforms of the
three $\mathbb{P}_{i}^{2}\,(i=1,2,3)$ under the blow-up $\pi_{2}:\widetilde{\mathcal{X}}\to\mathcal{X}'$
are $\mathbb{P}^{2}$ blown up at two points. 
\end{rem}
Making the blow-up $\widetilde{\mathcal{X}}\to\mathcal{X}=\mathcal{M}_{3,3}^{Loc}$
at each singular points of $\mathcal{M}_{3,3}$, we obtain the resolution
$\widetilde{\mathcal{M}}_{SecP}$ of the partial resolution $\mathcal{M}_{SecP}\to\mathcal{M}_{3,3}$
in Proposition \ref{prop:MsecP-Resolution}. Note that, in the resolution
$\widetilde{\mathcal{M}}_{SecP}$ thus obtained, we have the resolution
$\widetilde{\mathcal{X}}$ (the left in Fig.1) at all six singular
points.

\subsection{Flipping the line $\ell$ in $\widetilde{\mathcal{X}}$ to $\mathbb{P}^{2}$ }

Recall that we have introduced the line $\ell=\mathbb{P}_{1}^{2}\cap\mathbb{P}_{2}^{2}\cap\mathbb{P}_{3}^{2}$
in $\mathcal{X}$. Correspondingly, we have $\widetilde{\ell}=\widetilde{\mathbb{P}}_{1}^{2}\cap\widetilde{\mathbb{P}}_{2}^{2}\cap\widetilde{\mathbb{P}}_{3}^{2}$
on $\widetilde{\mathcal{X}}$. Here and in what follows we put $\widetilde{\,\,\empty\,\,}$
to indicate the the strict transform of a subvariety of $\mathcal{X}'$.
We can also write $\ell=E_{x}\cap E_{y}\cap E_{z}$ and $\widetilde{\ell}=\widetilde{E}_{x}\cap\widetilde{E}_{y}\cap\widetilde{E}_{z}$
by Proposition \ref{prop:BlowUp-along-xyz}. Let $N_{\widetilde{\ell}/\widetilde{\mathcal{X}}}$
be the normal bundle of $\widetilde{\ell}$ in $\widetilde{\mathcal{X}}$. 
\begin{lem}
We have $N_{\widetilde{\ell}/\widetilde{\mathcal{X}}}\simeq\mathcal{O}_{\mP^{1}}(-1)^{\oplus3}$. \end{lem}
\begin{proof}
By Proposition \ref{prop:BlowUp-at-p1p2}, $\widetilde{E}_{x}$, $\widetilde{E}_{y}$
and $\widetilde{E}_{z}$ are smooth on $\widetilde{\mathcal{X}}$.
Since $\widetilde{\ell}=\widetilde{E}_{x}\cap\widetilde{E}_{y}\cap\widetilde{E}_{z}$,
we have only to show that $\widetilde{E}_{x}\cdot\widetilde{\ell}=\widetilde{E}_{y}\cdot\widetilde{\ell}=\widetilde{E}_{z}\cdot\widetilde{\ell}=-1$.
By symmetry, it suffices to show that $\widetilde{E}_{x}\cdot\widetilde{\ell}=-1$.
Since $\widetilde{E}_{x}\cap\widetilde{E}_{y}=\widetilde{\mP}_{2}^{2}$
and $\widetilde{\mP}_{2}^{2}\cap\widetilde{\mP}_{3}^{2}=\widetilde{\ell}$,
we have $\widetilde{E}_{x}\cdot\widetilde{\ell}=(\widetilde{\mP}_{2}^{2}\cdot\widetilde{\ell})_{\widetilde{E}_{y}}=(\widetilde{\ell}^{2})_{\widetilde{\mP}_{3}^{2}}$.
Note that $\widetilde{\mP}_{i}^{2}\,(i=1,2,3)$ is a $\mP^{2}$ blown-up
at two points and $\widetilde{\ell}$ is a $(-1)$-curve on $\widetilde{\mP}_{i}^{2}$.
Therefore $(\widetilde{\ell}^{2})_{\widetilde{\mP}_{3}^{2}}=-1$ as
claimed. \end{proof}
\begin{prop}
\label{prop:flip-ell-P2}There is a flip which transforms the line
$\widetilde{\ell}$ to $\mathbb{P}^{2}.$\end{prop}
\begin{proof}
Here we only consider analytically for simplicity. See the proof of
Theorem \ref{thm:Resolutions} for an algebraic construction of the
flip. Since $N_{\widetilde{\ell}/\widetilde{\mathcal{X}}}\simeq\mathcal{O}_{\mP^{1}}(-1)^{\oplus3}$,
by blowing-up along the line $\ell$, we obtain $\ell\times\mathbb{P}^{2}$
as the exceptional divisor. Contracting this to $\mathbb{P}^{2}$,
we obtain the flip (cf. Fig.1). 
\end{proof}
We denote by $\widetilde{\mathcal{X}}^{+}\to\mathcal{X}$ the resulting
resolution after the flip of the resolution $\widetilde{\mathcal{X}}\to\mathcal{X}=\mathcal{M}_{3,3}^{Loc}$.
\begin{prop}
\label{prop:resolution-MsecP2}Making resolutions $\widetilde{\mathcal{X}}\to\mathcal{X}$
or $\widetilde{\mathcal{X}}^{+}\to\mathcal{X}$ locally at each of
six isomorphic singular points of $\mathcal{M}_{3,3}$, we obtain
the same resolution as the toric resolutions of $\mathcal{M}_{3,3}$
in Proposition \ref{prop:Toric-Res-MsecP} and Proposition \ref{prop:MsecP-Resolution}.\end{prop}
\begin{proof}
We verify the claim explicitly by writing the resolutions of $\mathcal{M}_{SecP}$
in Proposition \ref{prop:MsecP-Resolution}. Here we only sketch our
calculations. As described in the proof of Proposition \ref{prop:MsecP-Resolution},
the partial resolution $\mathcal{M}_{SecP}$ of $\mathcal{M}_{3,3}$
is covered by 108 affine charts, among which six charts are singular.
The singular charts are isomorphic to $\mathrm{Spec}\,\mathbb{C}[C_{NE}\cap L]$
which has two resolutions shown in Fig.1. By explicit calculations,
we find that 108 affine charts are grouped into six isomorphic blocks
of 18 charts (one singular and 17 smooth charts). We verify that each
block is isomorphic to $\widetilde{\mathcal{X}}$ or $\widetilde{\mathcal{X}}^{+}$
after making a resolution of the singular chart. 
\end{proof}
The above proposition provides us a global picture of the parameter
space of the GKZ $\mathcal{A}$-hypergeometric system in Proposition
\ref{prop:GKZ-sys-def}. Our task in the next section is to make a
covering of the parameter space $E(3,6)$ by certain Zariski open
subsets of the parameter space of the GKZ $\mathcal{A}$-hypergeometric
system.
\begin{rem}
Instead of constructing the resolution $\widetilde{\mathcal{X}}\to\mathcal{X}$
starting with the blow-up $\mathcal{X}'$ along $\Gamma$, we can
also make a resolution by first blowing-up along $z$-coordinate line
and then blowing-up along $x$- and $y$-coordinate lines. Since the
(strict transforms of) $x$- and $y$-coordinate lines are separated
by the first blowing-up along $z$-coordinate line, and the singularities
along these lines are of $A_{1}$ type, we obtain a resolution $\widehat{\mathcal{X}}\to\mathcal{X}$
in this way. Note that the resolution $\widehat{\mathcal{X}}\to\mathcal{X}$
introduces only three exceptional divisors from the blowing-ups, and
hence this is not isomorphic to $\widetilde{\mathcal{X}}\to\mathcal{X}$
in Proposition \ref{prop:BlowUp-at-p1p2} nor $\widetilde{\mathcal{X}}^{+}\to\mathcal{X}$.
Moreover, the generalized Frobenius method developed in \cite{HLY,HKTY}
does not apply to the resolution $\widehat{\mathcal{X}}\to\mathcal{X}$.
Recall that the generalized Frobenius method provides a closed formula
for the local solutions around special boundary points (LCSLs), such
as $o_{i}^{(k)}$ in $\widetilde{\mathcal{X}}\to\mathcal{X}$ or $\widetilde{\mathcal{X}}^{+}\to\mathcal{X}$,
given by normal crossing boundary or exceptional divisors. In the
resolution $\widehat{\mathcal{X}}\to\mathcal{X}$, there is no way
to have such special boundary points by the three exceptional divisors. 
\end{rem}
~

\section{\textbf{\textup{Blowing-up the Baily-Borel-Satake compactification
$\mathcal{M}_{6}$ \label{sec:BBS}}}}

We will study the relationship between the Baily-Borel-Satake compactification
$\mathcal{M}_{6}$ and the compactification $\mathcal{M}_{3,3}$,
which appears naturally from computing the period integrals. We recall
that the compactification $\mathcal{M}_{6}$ is birational to $\mathcal{M}_{3,3}$
with the birational map given by (\ref{eq:birat-map-M6-M33}). 

~

\subsection{Birational map $\phi:\mathcal{M}_{3,3}\protect\dashrightarrow\mathcal{M}_{6}$ }

Since both $\mathcal{M}_{6}$ and $\mathcal{M}_{3,3}$ have descriptions
in term of GIT quotients, the birational map $\phi$ can be given
explicitly by writing the relevant semi-invariants \cite{DoOrt,Reuv}.
We have sketched the results in Appendix \ref{sec:Appendix-birat-M};
in particular, we have given the explicit form of the birational map
using the (weighted-)homogeneous coordinates $[X_{0},X_{1},...,X_{5}]\in\mathbb{P}^{5}$
for $\mathcal{M}_{3,3}$ and $[Y_{0},...,Y_{4},Y_{5}]\in\mathbb{P}(1^{5},2)$
for $\mathcal{M}_{6}$. 
\begin{lem}
The following properties holds: \end{lem}
\begin{enumerate}
\item $\phi$ defines a map $\phi:\mathcal{M}_{3,3}\setminus\left\{ [1,1,1,1,1]\right\} \to\mathcal{M}_{6}$,
and 
\item $\phi^{-1}$ defines a map $\phi^{-1}:\mathcal{M}_{6}\setminus\left\{ Y_{0}=0\right\} \to\mathcal{M}_{3,3}$.
\end{enumerate}
~
\begin{prop}
Define the following divisor in $\mathcal{M}_{3,3}$: 
\begin{equation}
D_{0}=\left\{ X_{0}+X_{1}+X_{2}-X_{3}-X_{4}-X_{5}=0\right\} .\label{eq:divisor-D0}
\end{equation}
Then the birational map $\phi$ restricts to a 1 to 1 map 
\begin{equation}
\phi:\mathcal{M}_{3,3}\setminus D_{0}\to\phi(\mathcal{M}_{3,3}\setminus D_{0})\subset\mathcal{M}_{6}\label{eq:1to1-M33toM6}
\end{equation}
 to its image in $\mathcal{M}_{6}$.
\end{prop}
The proofs of the above lemma and proposition are easy from the explicit
forms (\ref{eq:YbyX-appndix}) and (\ref{eq:XbyY-appendix}) of the
birational maps $\phi$ and $\phi^{-1}$, respectively. Further properties,
e.g., the restriction $\phi:D_{0}\setminus\left\{ [1,1,1,1,1]\right\} \to\mathcal{M}_{6}$,
can be worked out, but we leave these to the reader (see \cite[Sect.2.4]{Reuv}).

\subsection{Singularities of $\mathcal{M}_{6}$ }

Singularities of $\mathcal{M}_{6}$ are well-studied objects in the
literatures (see \cite{YoshidaEtal,Hunt} for example). Here we summarize
the results from our viewpoints and using the (weighted-)homogeneous
coordinate $[Y_{0},Y_{1},...,Y_{5}]$ of $\mathbb{P}(1^{5},2)$. 
\begin{prop}
The variety $\mathcal{M}_{6}$ is singular along 15 lines which intersect
at 15 points which, respectively, correspond to one dimensional boundary
components and zero dimensional boundary components in the Baily-Borel-Satake
compactification. These 15 lines are located in $\left\{ Y_{5}=0\right\} \simeq\mathbb{P}^{4}\subset\mathbb{P}(1^{5},2)$. \end{prop}
\begin{proof}
The results are well-known in the literatures (see \cite{vanDerGeer,Hunt}
for example). An explicit description of the boundary components is
given in Appendix \ref{sec:appendix-Singular-L}. \end{proof}
\begin{prop}
Each of the 15 points of singularities is given by the intersection
of corresponding three lines. Vice versa, each of the 15 lines contains
three intersection points with other two lines at each intersection. \end{prop}
\begin{proof}
We verify the claimed properties using the equations for the 15 lines
in Appendix \ref{sec:appendix-Singular-L} and schematic description
of the 15 lines in Fig.5.
\end{proof}
~

\begin{figure}[h] 
\resizebox{11cm}{!}{
\includegraphics[trim=0cm 8cm 1cm 1cm,
clip=true]{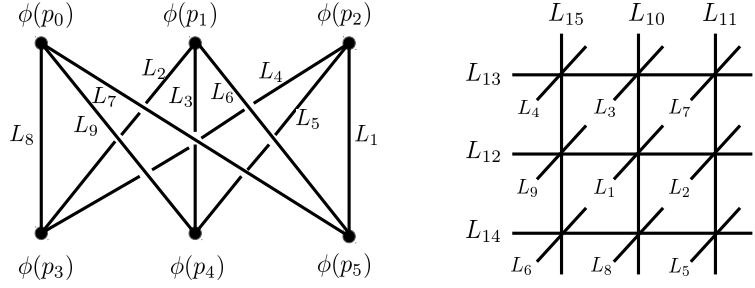}}
\caption{{\bf Fig.5} Configuration of the 15 lines. $\phi(p_k)$ represent the images of the coordinate points $p_k$ in ${\mathcal M}_{3,3}$. Lines $L_k (k=1,...,15)$ are given explicitly in Appendix \ref{sec:appendix-Singular-L}. $L_1,...,L_9$ shown in the left correspond to the 9 singular lines in ${\mathcal M}_{3,3}$. Lines  $L_{10},...,L_{15}$ are in the divisor $\{Y_0=0\}$ where $\phi^{-1}$ is not defined, and intersect with $L_1,...,L_9$ at one point as shown in the right. } 
\end{figure}
\begin{prop}
The 9 lines of singularities in $\mathcal{M}_{3,3}$ described in
Proposition \ref{prop:lines-in-M33} correspond to 9 of 15 lines in
$\mathcal{M}_{6}$ by the birational map $\phi:\mathcal{M}_{3,3}\dashrightarrow\mathcal{M}_{6}$.
In particular, the local structure $\mathcal{M}_{3,3}^{Loc}$ near
the 6 point is isomorphically mapped to the corresponding intersection
points of lines in $\mathcal{M}_{6}$. \end{prop}
\begin{proof}
Recall that the 9 lines in $\mathcal{M}_{6}$ come from coordinate
lines of $\mathbb{P}^{5}$ and intersect at 6 coordinate points. None
of the 9 lines nor their intersection points are contained in $D_{0}$
(\ref{eq:divisor-D0}). Hence these lines determine the corresponding
lines in $\mathcal{M}_{6}$ under the birational map $\phi$, along
which $\mathcal{M}_{6}$ is singular. Also the local structure $\mathcal{M}_{3,3}^{Loc}$
is mapped isomorphically to $\mathcal{M}_{6}$. 
\end{proof}
In the next subsection, we will see that the local structure near
all the 15 singular points in $\mathcal{M}_{6}$ are isomorphic to
$\mathcal{M}_{3,3}^{Loc}$. 

~

\subsection{$S_{6}$ action on $\mathcal{M}_{6}$}

Now recall that the homogeneous coordinate $Y_{i}$ is related to
the $3\times6$ matrix $A$ by (\ref{eq:Ys-appendix}). We note that
there is a natural action of the symmetric group $S_{6}$ sending
$A\to A\sigma:=A\rho(\sigma)$ by the permutation matrix $\rho(\sigma)$
representing $\sigma\in S_{6}$. This naturally induces linear actions
on homogeneous coordinates $Y_{i}(A)\mapsto Y_{i}(A\sigma)$. 
\begin{lem}
\label{lem:linear-action-Yi}The action $Y_{i}=Y_{i}(A)\mapsto Y_{i}(A\sigma)$
is linear and preserves the homogeneous weights of the coordinate
$[Y_{0},...,Y_{4},Y_{5}]\in\mathbb{P}(1^{5},2)$. \end{lem}
\begin{proof}
The claim is clear since $Y_{i}=Y_{i}(A)$ are generators of the semi-invariants
of $GL(3,\mathbb{C})$ of fixed degrees, and $Y_{i}(A\sigma)$ are
semi-invariants of the same degree with $Y_{i}(A)$.
\end{proof}
Geometric meaning of the action $A\to A\sigma$ is simply that it
changes the order of the (ordered) six points in $\mathbb{P}^{2}$.
From Lemma \ref{lem:linear-action-Yi}, it is easy to deduce the following
proposition.
\begin{prop}
The linear action $Y_{i}=Y_{i}(A)\mapsto Y_{i}(A\sigma)$ naturally
defines the corresponding automorphism $\psi(\sigma):\mathcal{M}_{6}\simeq\mathcal{M}_{6}$
for $\sigma\in S_{6}$. 
\end{prop}
We combine this isomorphism with the birational map $\phi$:~$\mathcal{M}_{3,3}\dashrightarrow\mathcal{M}_{6}$.
\begin{defn}
\label{def:phi-sigma}For $\sigma\in S_{6}$, we define the following
composite of $\psi(\sigma$) and $\phi$: 
\[
\phi_{\sigma}:\mathcal{M}_{3,3}\dashrightarrow\mathcal{M}_{6}\overset{\psi(\sigma)}{\simeq}\mathcal{M}_{6}.
\]
 ~
\end{defn}

\subsection{Covering $\mathcal{M}_{6}$ by open sets of toric varieties }

We now combine all the results about the moduli spaces $\mathcal{M}_{3,3}$
and $\mathcal{M}_{6}$. We first recall that $\mathcal{M}_{6}$ is
given by a hypersurface in $\mathbb{P}(1^{5},2)$. 
\begin{lem}
\label{lem:missing-point} The hypersurface $\mathcal{M}_{6}$ misses
the point $[0,0,0,0,0,1]\in\mathbb{P}(1^{5},2)$.\end{lem}
\begin{proof}
We simply verify the property from the definition (\ref{eq:DoubleCover}).\end{proof}
\begin{lem}
\label{lem:five-sigma}Take the following permutations $\sigma\in S_{6}$
\[
e,(34),(35),(24),(25)
\]
and name these by $\sigma_{k}\,(k=0,1,...,4)$ in order. Then under
the automorphism $\psi(\sigma_{k}):\mathcal{M}_{6}\simeq\mathcal{M}_{6}$,
the hyperplane $\left\{ Y_{0}=0\right\} \subset\mathcal{M}_{6}$ transforms
to $\left\{ Y_{k}=0\right\} \subset\mathcal{M}_{6}$ for $k=0,1,...,4$,
respectively. \end{lem}
\begin{proof}
By lemma \ref{lem:linear-action-Yi}, $Y(A\sigma_{k})$ is linear
in $Y_{i}$'s. We derive the claimed results by calculating the semi-invariants
given in (\ref{eq:Ys-appendix}) under the permutations.\end{proof}
\begin{prop}
\label{prop:Covering-M6}The moduli space $\mathbb{\mathcal{M}}_{6}$
is covered by copies of $\mathcal{M}_{3,3}\setminus D_{0}$. More
precisely, we have 
\[
\mathcal{M}_{6}=\bigcup_{k=0}^{4}\phi_{\sigma_{k}}\left(\mathcal{M}_{3,3}\setminus D_{0}\right).
\]
\end{prop}
\begin{proof}
By Lemma \ref{lem:missing-point}, one of the homogeneous coordinate
$Y_{0},...,Y_{4}$ does not vanish for any point of $\mathcal{M}_{6}$.
Then, due to Lemma \ref{lem:five-sigma}, any point is contained in
the union of the isomorphic images $\phi_{\sigma_{k}}\left(\mathcal{M}_{3,3}\setminus D_{0}\right)$
of $\mathcal{M}_{3,3}\setminus D_{0}$ (see (\ref{eq:1to1-M33toM6})
and Definition \ref{def:phi-sigma}).
\end{proof}
The local structures near each of the 15 singular points in $\mathcal{M}_{6}$
is isomorphic to the local structure of $\mathcal{M}_{3,3}^{Loc}$.
Making the resolution $\widetilde{\mathcal{X}}\to\mathcal{X}=\mathcal{M}_{3,3}^{Loc}$
given in Proposition \ref{prop:BlowUp-at-p1p2} at each singular point,
we have the resolution $\widetilde{\mathcal{M}}_{6}\to\mathcal{M}_{6}$.
Namely, let $f_{1}\colon\mathcal{M}'_{6}\to\mathcal{M}_{6}$ be the
blow-up along $\mathrm{Sing}\,\mathcal{M}_{6}$, which is the union
of $15$ lines. Then, $\mathcal{M}'_{6}$ has $2\times15$ singular
points. Let $f_{2}\colon\widetilde{\mathcal{M}}_{6}\to\mathcal{M}'_{6}$
be the blow-up at all the singular points.

Recall that locally we have another resolution $\widetilde{\mathcal{X}}^{+}$.
In the following theorem, we can globalize this to another resolution
of $\mathcal{M}_{6}$ connected with $\widetilde{\mathcal{M}}_{6}$
by a $4$-dimensional flip.
\begin{thm}
\label{thm:Resolutions} There exists another resolution $\widetilde{\mathcal{M}}_{6}^{+}$
of $\mathcal{M}_{6}$ which is connected with $\widetilde{\mathcal{M}}_{6}$
by a $4$-dimensional flip.\end{thm}
\begin{proof}
We have already constructed the flip $\widetilde{\mathcal{X}}\dashrightarrow\widetilde{\mathcal{X}}^{+}$
of $\widetilde{\ell}$ locally analytically in Proposition \ref{prop:flip-ell-P2}.
Let $\widetilde{\ell}_{1},\dots,\widetilde{\ell}_{15}\subset\widetilde{\mathcal{M}}_{6}$
be the copies of $\mP^{1}$ over the fifteen singular points of $\mathcal{M}_{6}$.
The remaining problem is to construct the flips of $\widetilde{\ell}_{1},\dots,\widetilde{\ell}_{15}$
algebraically and globally. The following properties guarantee this.
We will prove them in Appendix \ref{sec:Appendix-Resolution}.

Let $E$ be the $f_{1}$-exceptional divisor and $\widetilde{E}$
its strict transform on $\widetilde{\mathcal{M}}_{6}$. Set $f:=f_{1}\circ f_{2}$.
Then 

\begin{enumerate}

\item[(1)] $-(K_{\widetilde{\mathcal{M}}_{6}}+1/3\widetilde{E})$
is $f$-nef, and is numerically $f$-trivial only for $\widetilde{\ell}_{1},\dots,\widetilde{\ell}_{15}$.

\item[(2)] There exists a small contraction $\rho\colon\widetilde{\mathcal{M}}_{6}\to\overline{\mathcal{M}}_{6}$
over $\mathcal{M}_{6}$ contracting exactly $\widetilde{\ell}_{1}\cup\cdots\cup\widetilde{\ell}_{15}$.

\item[(3)] The contraction $\rho$ is a log flipping contraction
with respect to some klt pair $(\widetilde{\mathcal{M}}_{6},D)$.

\item[(4)] The flip $\widetilde{\mathcal{M}}_{6}\dashrightarrow\widetilde{\mathcal{M}}_{6}^{+}$
of $\rho$ exists and it coincides locally with the flip constructed
as in Proposition \ref{prop:flip-ell-P2}.

\end{enumerate}

This completes the construction of the resolution. \end{proof}
\begin{rem}
\label{rem:Lambda-M6-tM6} By Theorem \ref{thm:Resolutions}, we have
two algebraic resolutions $\widetilde{\mathcal{M}}_{6}\to\mathcal{M}_{6}$
and $\widetilde{\mathcal{M}}_{6}^{+}\to\mathcal{M}_{6}$, which are
related by a four dimensional flip. Interestingly, it will turn out
in \cite{HLTYpartII} that these two possibilities of algebraic resolutions
result in two non-isomorphic definitions of the lambda functions $\lambda_{K3}$
on $\mathcal{M}_{6}$.

~

~
\end{rem}

\section{\textbf{\textup{\label{sec:HyperG-D-Grassmann}Hypergeometric $\mathcal{D}$-modules
on Grassmannians}}}

In this section, we combine the results of earlier sections to give
our main results of this paper. 

We have obtained a global picture for the moduli space $\mathcal{M}_{6}$
in terms of the toric variety $\mathcal{M}_{3,3}$ which is closely
related to the toric variety $\mathcal{M}_{SecP}$. With these results
in hand, we now look at the hypergeometric system $E(3,6)$ defined
on its parameter space $\mathcal{M}_{6}$. To have a global picture,
it is better think of $E(3,6)$ as the corresponding $\mathcal{D}$-module
on $\mathcal{M}_{6}$. In this language, our first result is 
\begin{thm}
\label{thm:Thm1}On each of the open set $\phi_{\sigma_{k}}\left(\mathcal{M}_{3,3}\setminus D_{0}\right)$
$(k=0,1,..,4)$ of $\mathcal{M}_{6}$, the hypergeometric $\mathcal{D}$-module
of $E(3,6)$ restricts to the $\mathcal{D}$-module of the GKZ system
in Proposition \ref{prop:GKZ-sys-def}.\end{thm}
\begin{proof}
This follows by combining the results in Sections \ref{sec:GKZ-from-E36}
and \ref{sec:M33-from-Periodint} with Proposition \ref{prop:Covering-M6}.
\end{proof}
The GKZ $\mathcal{A}$-hypergeometric system has the natural compactification
$\mathcal{M}_{SecP}$ in terms of the secondary fan. As we saw in
Proposition \ref{prop:LCSL-GKZ}, the special boundary points (LCSLs)
arise in the resolutions of $\mathcal{M}_{SecP}$. By Propositions
\ref{prop:Toric-Res-MsecP} and \ref{prop:resolution-MsecP2}, the
resolutions of $\mathcal{M}_{SecP}$ are in fact the resolutions of
$\mathcal{M}_{3,3}$, and are given by the resolutions of the local
singularity $\widetilde{\mathcal{X}}\to\mathcal{M}_{3,3}^{Loc}$.
We have transformed these local structures to $\mathcal{M}_{6}$ by
the isomorphisms $\phi_{k}:\mathcal{M}_{3,3}\setminus D_{0}\simeq\phi_{\sigma_{k}}\left(\mathcal{M}_{3,3}\setminus D_{0}\right)$,
and obtained the desired resolutions of $\mathcal{M}_{6}$. Among
the resolutions, in particular, we have constructed two algebraic
resolutions $\widetilde{\mathcal{M}}_{6}$ and $\widetilde{\mathcal{M}}_{6}^{+}$.
Our second result is about the LCSLs in these resolutions.
\begin{thm}
\label{thm:Thm2}In the above resolutions of $\mathcal{M}_{6}$, the
LCSLs are given by the intersections of normal crossing divisors,
which are given by isomorphic images under $\phi_{\sigma_{k}}\,(k=0,1,...,4)$
of the divisors of the blow-ups $\widetilde{\mathcal{X}}\to\mathcal{X}=\mathcal{M}_{3,3}^{Loc}$
or their flips $\widetilde{\mathcal{X}}^{+}\to\mathcal{X}$.\end{thm}
\begin{proof}
The claims are shown in Sections \ref{sec:MoreOn-M33} and \ref{sec:BBS}.
By Proposition \ref{prop:LCSL-GKZ} and Proposition \ref{prop:resolution-MsecP2},
the boundary points are in fact the desired LCSLs.
\end{proof}
In a companion paper \cite{HLTYpartII}, we will construct the so-called
mirror maps from the local solutions near each LCSL. The mirror maps
turn out to be generalizations of the classical $\lambda$-function
for the Legendre family of the elliptic curves. We will call these
new examples of mirror maps\textit{ $\lambda_{K3}$-functions.} Then,
Theorem \ref{thm:Thm2} implies that the $\lambda_{K3}$-functions
have nice $q$-expansions (Fourier expansions) at the boundary points
in the suitable resolutions of the Baily-Borel-Satake compactification
of the double cover family of K3 surfaces. As mentioned in Remark
\ref{rem:Lambda-M6-tM6}, it will turn out in \cite{HLTYpartII} that
there are two non-isomorphic definitions of $\lambda_{K3}$-functions
corresponding to the two algebraic resolutions $\widetilde{\mathcal{M}}_{6}$
and $\widetilde{\mathcal{M}}_{6}^{+}$. 

\vskip0.1cm
\begin{rem}
For the double cover family of K3 surfaces, the two basically different
definitions of the moduli space are isomorphic; i.e., one is the GIT
compactification of the configurations of six lines, and the other
is the Baily-Borel-Satake compactification of the lattice polarized
K3 surfaces. Due to this nice property, we can associate geometry
to each point in the moduli space $\mathcal{M}_{6}$. We expect that
a nice geometry of degenerations, e.g. \cite{GS1,GS2}, will arise
from the boundary points which we have constructed in the resolutions
of $\mathcal{M}_{6}$. In particular, it is an interesting problem
to see how the geometry of the geometric mirror symmetry due to Strominger-Yau-Zaslow
\cite{SYZ} (and also \cite{GS1}) appears near these boundary points.
We should note here that the standard mirror symmetry for the lattice
polarized K3 surfaces \cite{DoMs} does not apply to the double cover
family of K3 surfaces because the transcendental lattice contains
$U(2)$ instead of $U$ (cf. \cite{GW}). 

\vskip0.2cm
\end{rem}
Finally, we note that the hypergeometric system $E(3,6)$ is a special
case of Aomoto-Gel'fand systems, which are called hypergeometric system
$E(n,m)$ on Grassmannians $G(n,m)$ (see e.g. \cite{Ao-Kita} and
refereces therein). Our theorems above are based on explicit constructions
for the case of $E(3,6)$, but we expect that they are generalized
in the following form:
\begin{conjecture}
Hypergeometric $\mathcal{D}$-modules of $E(n,m)$ on Grassmannians
have similar coverings by the $\mathcal{D}$-modules of suitable GKZ
systems. Namely, the parameter space of the system $E(n,m)$ has an
open covering by Zariski open subset of toric varieties on which the
system is represented (locally) by a $GKZ$ system. 
\end{conjecture}
The cases of $E(n,2n)$ are related to Calabi-Yau varieties which
are given by (suitable resolutions of) the double coverings of $\mathbb{P}^{n-1}$
branched along general $n$-hyperplanes. In particular, the case of
$E(4,8)$ and its related algebraic geometry has been worked in the
literatures \cite{GSvSZuo,ShZuo}. In this case, the GIT quotient
parameter space for $E(4,8)$ and its toric covering by $\mathcal{M}_{SecP}$
for the GKZ system become much more complicated. However, we expect
similar results as in Theorems \ref{thm:Thm1},\ref{thm:Thm2} hold
in general. 

\clearpage{}

\appendix
\renewcommand{\themyparagraph}{{\Alph{section}.\arabic{subsection}.\alph{myparagraph}}}

\section{\textbf{\textup{Six singular vertices of $Sec(\mathcal{A})$ \label{sec:Appendix-SecA}}}}

The secondary polytope $Sec(\mathcal{A})$ is defined for the Gorenstein
cone $Cone(\mathcal{A})$ generated by primitive generators $\mathcal{A}=\left\{ v_{1},v_{2},...,v_{9}\right\} $
given in (\ref{eq:defA}). We first consider all possible (regular)
triangulations $\mathcal{T}$ of the convex hull $Conv(\mathcal{A})$.
Each triangulation $T=\left\{ \sigma\right\} $ consists of simplices
$\sigma$, each of which corresponds to a simplicial cone in $Cone(\mathcal{A})$.
For a triangulation $T=\left\{ \sigma\right\} $, we set 
\[
\psi_{T}=\bigg(\sum_{v_{1}\prec\sigma}vol(\sigma),\sum_{v_{2}\prec\sigma}vol(\sigma),\cdots,\sum_{v_{9}\prec\sigma}vol(\sigma)\bigg)\in\mathbb{Z}^{9}.
\]
Here $vol(\sigma)$ is the volume of $\sigma$ normalized so that
the elementary simplex in $\mathbb{R}^{n}$ is $1$. The secondary
polytope is defined to be the convex hull $Conv(\left\{ \psi_{T}\right\} _{T\in\mathcal{T}})$
in $\mathbb{R}^{9}$. By translating one vertex, say $\psi_{T_{1}}$,
to the origin, this polytope now sits in $L_{\mathbb{R}}$ as introduced
in Subsection (\ref{para:SecPolytope}). There are 108 triangulations
for $Conv(\mathcal{A})$. Of those exactly six triangulations $T_{1},T_{2},...,T_{6}$
correspond to singularities in the compactification $\mathcal{M}_{SecP}=\mathbb{P}_{Sec(\mathcal{A})}$.
Below we list the all six vertices $\psi_{T_{i}}-\psi_{T_{1}}\in L$
for the convex hull;
\[
\begin{matrix}\begin{aligned}\psi_{T_{1}}-\psi_{T_{1}} & =4\,\,(\;\;\;0,\;\;\;0,\;\;\;0,\;0,\;0,\;0,\;0,\;0,\;0),\\
\psi_{T_{2}}-\psi_{T_{1}} & =4\,\,(-1,\;\;\;0,-1,\;0,\;1,\;0,\;0,\;1,\;0),\\
\psi_{T_{3}}-\psi_{T_{1}} & =4\,\,(-1,-1,\;\;\;0,\;1,\;0,\;0,\;1,\;0,\;0),\\
\psi_{T_{4}}-\psi_{T_{1}} & =4\,\,(\;\;\;0,-1,-1,\;0,\;0,\;1,\;0,\;0,\;1),\\
\psi_{T_{5}}-\psi_{T_{1}} & =4\,\,(-1,-1,-1,\;1,\;0,\;1,\;0,\;1,\;0),\\
\psi_{T_{6}}-\psi_{T_{1}} & =4\,\,(-1,-1,-1,\;0,\;1,\;0,\;1,\;0,\;1).
\end{aligned}
\end{matrix}
\]
The factor 4 is irrelevant to define toric variety from the convex
hull. Put 
\[
\mathcal{V}:=\left\{ \frac{1}{4}(\psi_{T_{i}}-\psi_{T_{1}})\mid i=1,...,6\right\} .
\]
Note that the set $\mathcal{V}\setminus\{0\}$ represents exactly
the exponents of $x,y,z,u,v$ in (\ref{eq:vars-xyzuv}). The cone
generated by $\mathcal{V}$ is $C_{0}$ given in (\ref{eq:cone-C0}),
while the cone generated by all 108 vertices is $C_{NE}$ given in
Proposition \ref{prop:SecP}, i.e., 
\[
C_{NE}=\sum_{i=1}^{108}\mathbb{R}_{\geq0}(\psi_{T_{i}}-\psi_{T_{1}}).
\]

~

\section{\textbf{\textup{Four dimensional cones $C_{0}$ and $C_{NE}$ \label{sec:AppendixB}}}}

Let $L=\mathrm{Ker}\left\{ \varphi_{\mathcal{A}}:\mathbb{Z}^{\mathcal{A}}\to\mathbb{Z}^{5}\right\} $
be the lattice defined in (\ref{eq:lattice-L}). Here, for convenience,
we summarize the data of the cones $C_{0},C_{NE}$ and their duals,
which are scattered in the text. We define a projection 
\begin{equation}
\pi_{4}:L\to\mZ^{4},\;\;\ell=(\ell_{1},...,\ell_{3},\ell_{4},...,\ell_{7},...,\ell_{9})\mapsto(\ell_{4},...,\ell_{7}).\label{eq:Appendix-Pi4}
\end{equation}
 It is an easy exercise to verify that $\pi_{4}:L\to\mathbb{Z}^{4}$
is an isomorphism. In this paper we shall often use $\pi_{4}$ to
represent vertices in $L$ as four component vectors for computations. 
\begin{prop}
The cones $C_{0}\subset C_{NE}$ and $C_{NE}$ in $L_{\mathbb{R}}$
are written under the above identification by
\[
\begin{matrix}\begin{aligned} & C_{0}=\mathrm{Cone}\left\{ (0,1,0,0),(1,0,0,1),(0,0,1,0),(1,0,1,0),(0,1,0,1)\right\} ,\\
 & C_{NE}={\rm Cone}\left\{ \begin{matrix}(1,0,0,0),(0,1,-1,1),(1,-1,1,0),(0,0,0,1)\\
(0,0,1,-1),(-1,1,0,0)
\end{matrix}\right\} .
\end{aligned}
\end{matrix}
\]

\end{prop}
It is straightforward to verify the following results from explicit
calculations.
\begin{prop}
The dual cones $C_{0}^{\vee}\supset C_{NE}^{\vee}$ are written by
the following primitive generators;
\[
\begin{matrix}\begin{aligned} & C_{0}^{\vee}=\mathrm{Cone}\left\{ \begin{matrix}(0,1,0,0),(1,0,0,0),(0,0,1,0),(0,0,0,1)\\
(1,1,0,-1),(-1,0,1,1)
\end{matrix}\right\} ,\\
 & C_{NE}^{\vee}=\mathrm{Cone}\left\{ (1,1,0,0),(0,1,1,0),(0,0,1,1),(1,1,1,0),(0,1,1,1)\right\} .
\end{aligned}
\end{matrix}
\]
The dual cone $C_{0}^{\vee}$ is a Gorenstein cone, while $C_{NE}^{\vee}$
is not. 

~
\end{prop}

\section{\textbf{\textup{Picard-Fuchs operators on ${\rm Spec}\,\protect\mC[(\sigma_{2}^{(1)})^{\vee}\cap L]$
\label{sec:Appendix-PF-eqs}}}}

We list the Picard-Fuchs differential operators discussed in Subsection
\ref{sub:GKZ-LCSLs} following the notation there. A complete set
of differential operators $\mathcal{D}_{\ell}$ are given by the following
$\ell$'s:
\[
\begin{matrix}\ell^{(1)},\;\ell^{(2)},\;\ell^{(3)},\;\ell^{(1)}+\ell^{(4)},\;\ell^{(2)}+\ell^{(4)},\;\ell^{(3)}+\ell^{(4)},\\
\ell^{(1)}+\ell^{(2)}+\ell^{(4)},\;\ell^{(1)}+\ell^{(3)}+\ell^{(4)},\;\ell^{(2)}+\ell^{(3)}+\ell^{(4)}.
\end{matrix}
\]
We name by $\mathcal{D}_{i}\,(i=1,...,9)$ the associated operators
$\mathcal{D}_{\ell}$ in the above order of $\ell$ with setting $z_{i}:=\mathtt{a}^{\ell^{(i)}}$
and $\theta_{i}:=z_{i}\frac{\partial}{\partial z_{i}}$. They take
the following forms: 
\[
\begin{matrix}\mathcal{D}_{1}=(\theta_{1}+\theta_{2}-\theta_{4})(\theta_{1}+\theta_{3}-\theta_{4})+z_{1}(\theta_{1}+\frac{1}{2})(\theta_{1}-\theta_{4}),\\
\mathcal{D}_{2}=(\theta_{1}+\theta_{2}-\theta_{4})(\theta_{2}+\theta_{3}-\theta_{4})+z_{2}(\theta_{2}+\frac{1}{2})(\theta_{2}-\theta_{4}),\\
\mathcal{D}_{3}=(\theta_{1}+\theta_{3}-\theta_{4})(\theta_{2}+\theta_{3}-\theta_{4})+z_{3}(\theta_{3}+\frac{1}{2})(\theta_{3}-\theta_{4}),\\
\mathcal{D}_{4}=(\theta_{2}-\theta_{4})(\theta_{3}-\theta_{4})-z_{1}z_{4}(\theta_{1}+\frac{1}{2})(\theta_{2}+\theta_{3}-\theta_{4}),\;\;\;\;\\
\mathcal{D}_{5}=(\theta_{1}-\theta_{4})(\theta_{3}-\theta_{4})-z_{2}z_{4}(\theta_{2}+\frac{1}{2})(\theta_{1}+\theta_{3}-\theta_{4}),\;\;\;\;\\
\mathcal{D}_{6}=(\theta_{1}-\theta_{4})(\theta_{2}-\theta_{4})-z_{3}z_{4}(\theta_{3}+\frac{1}{2})(\theta_{1}+\theta_{2}-\theta_{4}),\;\;\;\;\\
\mathcal{D}_{7}=(\theta_{1}+\theta_{2}-\theta_{4})(\theta_{3}-\theta_{4})+z_{1}z_{2}z_{4}(\theta_{1}+\frac{1}{2})(\theta_{2}+\frac{1}{2}),\;\;\\
\mathcal{D}_{8}=(\theta_{1}+\theta_{3}-\theta_{4})(\theta_{2}-\theta_{4})+z_{1}z_{3}z_{4}(\theta_{1}+\frac{1}{2})(\theta_{3}+\frac{1}{2}),\;\;\\
\mathcal{D}_{9}=(\theta_{2}+\theta_{3}-\theta_{4})(\theta_{1}-\theta_{4})+z_{2}z_{3}z_{4}(\theta_{2}+\frac{1}{2})(\theta_{3}+\frac{1}{2}).\;\;
\end{matrix}
\]
The radical $\sqrt{dis}$ of the discriminant is given by 
\[
\begin{matrix}\begin{aligned}z_{1}z_{2}z_{3}z_{4}\times\prod_{i=1}^{3}(1+z_{i})(1+z_{i}z_{4})\times\prod_{1\leq i<j\leq3}(1-z_{i}z_{j}z_{4})\\
\;\;\times\left(1-(z_{1}z_{2}+z_{1}z_{3}+z_{2}z_{3}+z_{1}z_{2}z_{3})z_{4}-z_{1}z_{2}z_{3}z_{4}^{2}\right).
\end{aligned}
\end{matrix}
\]

~

\specialsection{\textbf{\label{sec:Appendix-birat-M}Birational map $\phi:\mathcal{M}_{3,3}\protect\dashrightarrow\mathcal{M}_{6}$}}

Here we describe the birational map $\phi:\mathcal{M}_{3,3}\dashrightarrow\mathcal{M}_{6}$
explicitly by coordinates. We follow the general definitions given
in \cite{DoOrt,Reuv}. 

\vskip0.3cm\noindent\textbf{C.1. Semi-invariants for $\mathcal{M}_{6}$.}
As in the text, let us consider an ordered configuration of six lines
$(\ell_{1}\ell_{2}...\ell_{6})$ by the corresponding sequence of
points $A=(\bm{a}_{1}\bm{a}_{2}...\bm{a}_{6})$ represented by a $3\times6$
matrix. Based on the classical invariant theory, following \cite{DoOrt},
we define the following homogeneous polynomials 
\begin{equation}
\begin{alignedat}{1}Y_{0}=Y_{0}(A) & =[1\,2\,3][4\,5\,6],\\
Y_{1}=Y_{1}(A) & =[1\,2\,4][3\,5\,6],\\
Y_{2}=Y_{2}(A) & =[1\,2\,5][3\,4\,6],\\
Y_{3}=Y_{3}(A) & =[1\,3\,4][2\,5\,6],\\
Y_{4}=Y_{4}(A) & =[1\,3\,5][2\,4\,6],\\
Y_{5}=Y_{5}(A) & =[1\,2\,3][1\,4\,5][2\,4\,6][3\,5\,6]-[1\,2\,4][1\,3\,5][2\,3\,6][4\,5\,6],
\end{alignedat}
\label{eq:Ys-appendix}
\end{equation}
where $[i\,j\,k]:=\det(\bm{a}_{i}\bm{a_{j}}\bm{a}_{k})$, and we count
the weight $Y_{0},...,Y_{4}$ by 1 and $Y_{5}$ by 2 since they are
sections of $\mathcal{L}$ and $\mathcal{L}^{\otimes2}$, respectively,
for a $GL(3,\mathbb{C})\times(\mathbb{C}^{*})^{6}$-equivariant line
bundle $\mathcal{L}$ with the fiber $\mathbb{C}_{det}\otimes\mathbb{C}_{(1^{6})}$.
The GIT quotient $GL(3,\mathbb{C})\setminus M(3,6)^{ss}/(\mathbb{C}^{*})^{6}$
coincides with the Zariski closure of the image $A\mapsto[Y_{0,}Y_{1},...,Y_{5}]$
in the weighted projective space $\mathbb{P}(1^{5},2)$, which we
have denoted by $\mathcal{M}_{6}$ in the text. 

From symmetry reason, we extend the weight one variables $Y_{0},...,Y_{4}$
to 
\[
\begin{matrix}Y_{6}=[1\,2\,6][3\,4\,5], & Y_{7}=[1\,3\,6][2\,4\,5], & Y_{8}=[1\,4\,6][2\,3\,5],\\
Y_{9}=[1\,5\,6][2\,3\,4], & Y_{10}=[1\,4\,5][2\,3\,6].
\end{matrix}
\]
These satisfy the following linear relations, which are nothing but
Pl\"ucker relations of the Grassmannian $G(3,6)$:
\begin{equation}
\begin{array}{cc}
Y_{0}-Y_{1}+Y_{2}-Y_{6}=0, & Y_{0}-Y_{6}+Y_{7}-Y_{10}=0,\\
Y_{2}-Y_{3}-Y_{7}+Y_{8}=0, & Y_{2}-Y_{3}-Y_{6}+Y_{9}=0,\\
Y_{3}-Y_{4}+Y_{6}+Y_{10}=0.
\end{array}\label{eq:AppC-rel-Ys}
\end{equation}

\vskip0.3cm\noindent\textbf{C.2. Semi-invariants for $\mathcal{M}_{3,3}$.}
When we write an ordered 6 lines in general position by $A=(\bm{a}_{1}\bm{a}_{2}...\bm{a}_{6})$
as above, the birational map (\ref{eq:birat-map-M6-M33}) may be expressed
by 
\[
A\mapsto A^{*}=(\bm{a}_{2}\times\bm{a}_{3}\,\bm{a}_{3}\times\bm{a}_{1}\,\bm{a}_{1}\times\bm{a}_{2}\,\,\bm{a}_{4}\,\,\bm{a}_{5}\,\,\bm{a}_{6})=:(\bm{c}_{1}\,\bm{c}_{2}\,\bm{c}_{3}\,\bm{a}_{4}\,\bm{a}_{5}\,\bm{a}_{6}),
\]
where $\bm{a}_{i}\times\bm{a}_{j}$ represents the exterior product
of two space vectors $\bm{a}_{i},\bm{a}_{j}$. Similarly to the case
of $A$, two algebraic groups $GL(3,\mathbb{C})$ and $(\mathbb{C}^{*})^{6}$
act on the column vectors of $A^{*}$, but with different representations.
This time, the semi-invariants are given by
\begin{equation}
\begin{alignedat}{1}X_{0}=X_{0}(A^{*}) & =(\bm{c}_{1},\bm{a}_{4})(\bm{c}_{2},\bm{a}_{5})(\bm{c}_{3},\bm{a}_{6}),\\
X_{1}=X_{1}(A^{*}) & =(\bm{c}_{1},\bm{a}_{5})(\bm{c}_{2},\bm{a}_{6})(\bm{c}_{3},\bm{a}_{4}),\\
X_{2}=X_{2}(A^{*}) & =(\bm{c}_{1},\bm{a}_{6})(\bm{c}_{2},\bm{a}_{4})(\bm{c}_{3},\bm{a}_{5}),\\
X_{3}=X_{3}(A^{*}) & =(\bm{c}_{1},\bm{a}_{4})(\bm{c}_{2},\bm{a}_{6})(\bm{c}_{3},\bm{a}_{5}),\\
X_{4}=X_{4}(A^{*}) & =(\bm{c}_{1},\bm{a}_{5})(\bm{c}_{2},\bm{a}_{4})(\bm{c}_{3},\bm{a}_{6}),\\
X_{5}=X_{5}(A^{*}) & =(\bm{c}_{1},\bm{a}_{6})(\bm{c}_{2},\bm{a}_{5})(\bm{c}_{3},\bm{a}_{4}),
\end{alignedat}
\label{eq:Xs-appendix2}
\end{equation}
with $(\bm{x},\bm{y}):=\sum_{i=1}^{3}x_{i}y_{i}$. Using these semi-invariants,
the GIT quotient of the configuration space of 3 points and 3 lines
in $\mathbb{P}^{2}$ coincides with the Zariski closure of the image
$A^{*}\mapsto[X_{0},X_{1},...,X_{5}]$ in $\mathbb{P}^{5}$, which
is the toric variety $\mathcal{M}_{3,3}$. 

\vskip0.3cm\noindent\textbf{C.3. The birational map $\phi:\mathcal{M}_{3,3}\dashrightarrow\mathcal{M}_{6}$
and $S_{6}$ actions. }The birational map (\ref{eq:birat-map-M6-M33})
can be written explicitly by eliminating the variables $\bm{a}_{i}$
from (\ref{eq:Ys-appendix}) and (\ref{eq:Xs-appendix2}). Using a
Gr\"obner basis package in symbolic manipulations, we obtain
\begin{equation}
\begin{alignedat}{1}Y_{0}= & X_{0}+X_{1}+X_{2}-X_{3}-X_{4}-X_{5},\\
Y_{1}= & X_{1}-X_{5},\\
Y_{2}= & X_{3}-X_{2},\\
Y_{3}= & X_{4}-X_{2},\\
Y_{4}= & X_{0}-X_{5},\\
Y_{5}= & X_{0}X_{1}+X_{0}X_{2}+X_{1}X_{2}-X_{3}X_{4}-X_{3}X_{5}-X_{4}X_{5},
\end{alignedat}
\label{eq:YbyX-appndix}
\end{equation}
which represents the birational map $\phi:\mathcal{M}_{3,3}\dashrightarrow\mathcal{M}_{6}$.
The inverse rational map $\phi^{-1}$ takes the following form:
\begin{equation}
\begin{alignedat}{1}X_{0}= & \frac{1}{2Y_{0}}\left(Y_{0}(\,\,Y_{0}-Y_{1}+Y_{2}+Y_{3}+Y_{4})-Y_{1}Y_{4}+Y_{2}Y_{3}+Y_{5}\right),\\
X_{1}= & \frac{1}{2Y_{0}}\left(Y_{0}(\,\,Y_{0}+Y_{1}+Y_{2}+Y_{3}-Y_{4})-Y_{1}Y_{4}+Y_{2}Y_{3}+Y_{5}\right),\\
X_{2}= & \frac{1}{2Y_{0}}\left(Y_{0}(-Y_{0}+Y_{1}-Y_{2}-Y_{3}+Y_{4})-Y_{1}Y_{4}+Y_{2}Y_{3}+Y_{5}\right),\\
X_{3}= & \frac{1}{2Y_{0}}\left(Y_{0}(-Y_{0}+Y_{1}+Y_{2}-Y_{3}+Y_{4})-Y_{1}Y_{4}+Y_{2}Y_{3}+Y_{5}\right),\\
X_{4}= & \frac{1}{2Y_{0}}\left(Y_{0}(-Y_{0}+Y_{1}-Y_{2}+Y_{3}+Y_{4})-Y_{1}Y_{4}+Y_{2}Y_{3}+Y_{5}\right),\\
X_{5}= & \frac{1}{2Y_{0}}\left(Y_{0}(\,\,Y_{0}-Y_{1}+Y_{2}+Y_{3}-Y_{4})-Y_{1}Y_{4}+Y_{2}Y_{3}+Y_{5}\right).
\end{alignedat}
\label{eq:XbyY-appendix}
\end{equation}

~

\section{\textbf{\textup{\label{sec:appendix-Singular-L}Singular lines in
$\mathcal{M}_{6}$}}}

Here, for convenience of the reader, we list the ideals for 15 lines
in $\mathcal{M}_{6}.$ Since all lines are in the hyperplane $Y_{5}=0$,
we omit $Y_{5}$ in each ideal.

\[
\begin{matrix}\left\langle Y_{2}-Y_{3},Y_{1}-Y_{4},Y_{0}+Y_{3}-Y_{4}\right\rangle , & \left\langle Y_{3},Y_{4},Y_{0}-Y_{1}+Y_{2}\right\rangle , & \left\langle Y_{2},Y_{4},Y_{0}-Y_{1}+Y_{3}\right\rangle ,\\
\left\langle Y_{1},Y_{4},Y_{0}+Y_{2}\right\rangle ,\qquad & \left\langle Y_{1},Y_{4},Y_{0}+Y_{3}\right\rangle ,\;\;\;\;\;\;\;\; & \left\langle Y_{2},Y_{3},Y_{0}-Y_{1}\right\rangle ,\qquad\;\\
\left\langle Y_{2},Y_{3},Y_{0}-Y_{4}\right\rangle ,\qquad & \left\langle Y_{1},Y_{3},Y_{0}+Y_{2}-Y_{4}\right\rangle , & \left\langle Y_{1},Y_{2},Y_{0}+Y_{3}-Y_{4}\right\rangle ;\\
\left\langle Y_{0},Y_{1}-Y_{3},Y_{2}-Y_{4}\right\rangle , & \left\langle Y_{0},Y_{3},Y_{4}\right\rangle , & \left\langle Y_{0},Y_{1}-Y_{2},Y_{3}-Y_{4}\right\rangle ,\\
\left\langle Y_{0},Y_{2},Y_{4}\right\rangle ,\qquad & \left\langle Y_{0},Y_{1},Y_{3}\right\rangle , & \left\langle Y_{0},Y_{1},Y_{2}\right\rangle .
\end{matrix}
\]
 We write these lines by $L_{1},...,L_{9};L_{10},...,L_{15}$ in order.
The first 9 lines correspond to the singular lines in $\mathcal{M}_{3,3}$
under the birational map $\phi:\mathcal{M}_{3,3}\dashrightarrow\mathcal{M}_{6}$.
As for the last 6 lines, which lie on $\left\{ Y_{0}=0\right\} $,
we can verify that the inverse images of these lines are planes in
$\mathcal{M}_{3,3}$ which are given by 
\[
P_{ijk}=\left\{ X_{0}=X_{i},\,X_{1}=X_{j},\,X_{2}=X_{k}\right\} \subset\mathcal{M}_{3,3}
\]
for 6 permutations $(ijk)$ of $(3,4,5)$. 

~

~

\section{\textbf{\textup{\label{sec:Appendix-Resolution}Properties used in
the proof of Theorem \ref{thm:Resolutions}}}}

We prove the properties used in the proof of Theorem \ref{thm:Resolutions}.
We continue using the notation introduced there. 

\begin{cla}\label{cla:flip} The following assertions hold: 

\begin{enumerate}

\item[(1)] $-(K_{\widetilde{\mathcal{M}}_{6}}+1/3\widetilde{E})$
is $f$-nef, and is numerically $f$-trivial only for $\widetilde{\ell}_{1},\dots,\widetilde{\ell}_{15}$.

\item[(2)] There exists a small contraction $\rho\colon\widetilde{\mathcal{M}}_{6}\to\overline{\mathcal{M}}_{6}$
over $\mathcal{M}_{6}$ contracting exactly $\widetilde{\ell}_{1}\cup\cdots\cup\widetilde{\ell}_{15}$.

\item[(3)] The contraction $\rho$ is a log flipping contraction
with respect to some klt pair $(\widetilde{\mathcal{M}}_{6},D)$.

\item[(4)] The flip $\widetilde{\mathcal{M}}_{6}\dashrightarrow\widetilde{\mathcal{M}}_{6}^{+}$
of $\rho$ exists and it coincides locally with the flip constructed
as in Proposition \ref{prop:lines-in-M33}.

\end{enumerate}

\end{cla}
\begin{proof}
(1) Note that $K_{\mathcal{M}'_{6}}=f_{1}^{*}K_{\mathcal{M}_{6}}+E$
since $f_{1}$ is the blow-up along $\mathrm{Sing}\,\mathcal{M}_{6}$
and $\mathcal{M}_{6}$ has ODP generically along $\mathrm{Sing}\,\mathcal{M}_{6}$.
Let $F$ be the $f_{2}$-exceptional divisor. We have $K_{\widetilde{\mathcal{M}}_{6}}=f_{2}^{*}K_{\mathcal{M}_{6}}+F$
since $f_{2}$ is the blow-up at singular points isomorphic to the
vertex of the cone over the Segre $(\mP^{1})^{3}$. Therefore we have
$-(K_{\widetilde{\mathcal{M}}_{6}}+1/3\widetilde{E})=-(f_{2}^{*}f_{1}^{*}K_{\mathcal{M}_{6}}+1/3\widetilde{E}+f_{2}^{*}E+F)$.
Now note that $f_{2}^{*}E=\widetilde{E}+F$, which follows from the
local computations as in the proof of Proposition \ref{prop:BlowUp-along-xyz}
(note that, in the proof of Proposition \ref{prop:BlowUp-along-xyz},
we can read off that the divisor $E$ is defined by $u=0$ on the
chart of $\widetilde{\mathcal{X}}$ with $U=1$). Therefore we have
\[
-(K_{\widetilde{\mathcal{M}}_{6}}+1/3\widetilde{E})\equiv_{\mathcal{M}_{6}}-(4/3\widetilde{E}+2F).
\]
It is easy to see that $-(4/3\widetilde{E}+2F)$ is $f$-nef from
the local computations for $f_{1}$ and $f_{2}$.

(2) By Proposition \ref{prop:BlowUp-at-p1p2}, $(\widetilde{\mathcal{M}}_{6},1/3\widetilde{E})$
is a klt pair. Since $-(K_{\widetilde{\mathcal{X}}_{6}}+1/3\widetilde{E})$
is $f$-nef by (1), and also $f$-big, then $-(K_{\widetilde{\mathcal{M}}_{6}}+1/3\widetilde{E})$
is $f$-semiample by Kawamata-Shokurov's base point free theorem (\cite{KMM}).
Therefore, there exists a contraction $\rho\colon\widetilde{\mathcal{M}}_{6}\to\overline{\mathcal{M}}_{6}$
over $\mathcal{M}_{6}$ defined by a sufficient multiple of $-(K_{\widetilde{\mathcal{M}}_{6}}+1/3\widetilde{E})$.
Since $-(K_{\widetilde{\mathcal{M}}_{6}}+1/3\widetilde{E})$ is numerically
$f$-trivial only for $l_{1},\dots,l_{15}$ by (1), we see that $\rho$
is the desired contraction.

(3) The proof given here may look technical but more or less is standard
for experts. As we see in the proof of (1) and (2), $(\widetilde{\mathcal{M}}_{6},1/3\widetilde{E})$
is a klt pair such that $-(K_{\widetilde{\mathcal{M}}_{6}}+1/3\widetilde{E})$
is numerically $\rho$-trivial. Now let $A$, $B$ be ample divisors
on $\widetilde{\mathcal{M}}_{6}$ and $\overline{\mathcal{M}}_{6}$,
respectively. Then we see that $|m\rho^{*}B-A|\not=\emptyset$ for
$m\gg0$ since $\rho^{*}B$ is big. Let $G$ be a member of $|m\rho^{*}B-A|$.
Then $(\widetilde{\mathcal{M}}_{6},1/3\widetilde{E}+1/k\,G)$ is klt
for $k\gg0$ and $-(K_{\widetilde{\mathcal{M}}_{6}}+1/3\widetilde{E}+1/k\,G)$
is $\rho$-ample since $-(K_{\widetilde{\mathcal{M}}_{6}}+1/3\widetilde{E})$
is numerically $\rho$-trivial and $-G$ is $\rho$-ample. Setting
$D:=1/3\widetilde{E}+1/k\,G$, we obtain a desired log pair.

(4) The existence of the flip is a consequence of (3) and \cite[Cor.1.4.1]{BCHM}.
By the local uniquness of the flip \cite[Prop.5-11-1]{KMM}, it coincides
locally with the flip constructed as in Proposition \ref{prop:lines-in-M33}. 
\end{proof}
~

\vspace{1cm}

{\footnotesize{}Shinobu Hosono}{\footnotesize \par}

{\footnotesize{}Department of Mathematics, Gakushuin University, }{\footnotesize \par}

{\footnotesize{}Mejiro, Toshima-ku, Tokyo 171-8588, Japan }{\footnotesize \par}

{\footnotesize{}e-mail: hosono@math.gakushuin.ac.jp}{\footnotesize \par}

~

{\footnotesize{}Bong H. Lian}{\footnotesize \par}

{\footnotesize{}Department of Mathematics, Brandeis University, }{\footnotesize \par}

{\footnotesize{}Waltham MA 02454, U.S.A. }{\footnotesize \par}

{\footnotesize{}e-mail: lian@brandeis.edu}{\footnotesize \par}

~

{\footnotesize{}Hiromichi Takagi}{\footnotesize \par}

{\footnotesize{}Department of Mathematics, Gakushuin University, }{\footnotesize \par}

{\footnotesize{}Mejiro, Toshima-ku, Tokyo 171-8588, Japan }{\footnotesize \par}

{\footnotesize{}e-mail: hiromici@math.gakushuin.ac.jp}{\footnotesize \par}

~

{\footnotesize{}S.-T. Yau}{\footnotesize \par}

{\footnotesize{}Department of Mathematics, Harvard University, }{\footnotesize \par}

{\footnotesize{}Cambridge MA 02138, U.S.A. }{\footnotesize \par}

{\footnotesize{}e-mail: yau@math.harvard.edu}{\footnotesize \par}

\end{document}